\documentclass{ndjflart_arXiv}
\usepackage[T1]{fontenc}
\usepackage{tgtermes}
\usepackage{mathptmx}  
\usepackage[scaled=.92]{helvet}
\usepackage{amsthm,amsmath,amssymb}
\usepackage{mathrsfs}
\usepackage[numbers]{natbib}  
\usepackage[colorlinks,citecolor=blue,urlcolor=blue]{hyperref}  

\usepackage{enumerate}

\artstatus{am} 

\usepackage{amstext}

\usepackage{latexsym}
\usepackage{verbatim}
\usepackage{stmaryrd}
\usepackage{natbib}
\usepackage{relsize}
\usepackage{graphicx}
\usepackage{subfigure}
\usepackage{tikz}
\usepackage{tkz-graph}
\usetikzlibrary{shapes.geometric}
\usepackage{framed}
\usepackage{linguex}
\usepackage{multicol}
\usepackage{proof}

\tikzstyle{index on}=[inner sep=2pt, white, circle, fill=black]
\tikzstyle{index off}=[inner sep=2pt, black, circle, draw]
\tikzstyle{index gray}=[inner sep=2pt, black, circle, fill=lightgray]
\tikzstyle{opaque}=[fill=gray,fill opacity=.1]

\newtheorem{theorem}{Theorem}[section]
\newtheorem{lemma}[theorem]{Lemma}

\newtheorem{corollary}[theorem]{Corollary} 
\theoremstyle{definition}
\newtheorem{definition}[theorem]{Definition}

\theoremstyle{remark}

\newtheorem{example}[theorem]{Example}
\newtheorem{proposition}[theorem]{Proposition}

\renewcommand{\phi}{\varphi}

\newcommand\<{\langle}
\renewcommand\>{\rangle}

\renewcommand{\aa}{\alpha}
\renewcommand{\P}{\ensuremath{\mathcal{P}}}
\newcommand{\R}{\mathcal{R}}
\renewcommand{\L}{\ensuremath{\mathcal{L}}}

\renewcommand{\int}{\textsf{int}}

\newcommand{\dep}{\ensuremath{\mathop{=\!\!}}}

\newcommand{\Lp}{\ensuremath{\mathcal{L}}\xspace}
\newcommand{\Ld}{\ensuremath{\mathcal{L}_d}\xspace}

\newcommand{\cpl}{\textsf{CPL}\xspace}
\newcommand{\ipl}{\textsf{IPL}\xspace}
\newcommand{\s}{\textsf{S}\xspace}
\newcommand{\dn}{\textsf{DNE}}
\newcommand{\inqb}{\textsf{InqB}\xspace}
\newcommand{\lqd}{\ensuremath{\textsf{InqB}^\lor}\xspace}
\newcommand{\ilqd}{\textsf{InqI}\xspace}
\newcommand{\inqi}{\textsf{InqI}\xspace}
\newcommand{\inqim}{\ensuremath{\textsf{InqI}^-}\xspace}
\newcommand{\lqdp}{\ensuremath{\textsf{InqB}^\lor}\xspace}

\newcommand{\entilqd}{\ensuremath{\models_{\textsf{InqI}}}\xspace}
\newcommand{\prilqd}{\ensuremath{\vdash_{\textsf{InqI}}}\xspace}
\newcommand{\preqilqd}{\ensuremath{\dashv\vdash_{\textsf{InqI}}}\xspace}
\newcommand{\deq}{\overset{*}{=}}
\newcommand{\MT}{\ensuremath{\textsf{MT}_0}\xspace}
\newcommand{\Sf}{\textsf{S4}\xspace}

\def\DISJI{\rotatebox[origin=c]{-90}{$\!{\mathlarger{\mathlarger{\mathlarger{\mathlarger{\,\mathlarger{\geqslant}}}}}}$}}
\def\Disji{\rotatebox[origin=c]{-90}{$\!{\mathlarger{\mathlarger{\,\mathlarger{\geqslant}}}}$}}
\def\disji{\rotatebox[origin=c]{-90}{$\!{\geqslant}$}}
\newcommand{\lori}{\ensuremath{\,\disji\,}}
\newcommand{\Lori}{\,\Disji\,}

\renewcommand{\aa}{\alpha}
\newcommand{\bb}{\beta}

\startlocaldefs

\endlocaldefs

\begin{document}

\begin{frontmatter}

  \title{Questions and dependency in intuitionistic logic}
  \runtitle{Questions and dependency in intuitionistic logic}

  \author{\fnms{Ivano} 
    \snm{Ciardelli}
    \corref{}
    \ead[label=e1]{sobekal@hotmail.com}
  },
  \address{Institute for Logic, Language and Computation\\
    University of Amsterdam\\
    Science Park 107, 1098 XG Amsterdam\\
    The Netherlands\\
    \printead{e1}
  }%
  \author{\fnms{Rosalie}
    \snm{Iemhoff}\ead[label=e2]{R.Iemhoff@uu.nl}}
  \address{Department of Philosophy and Religious Studies\\
    Utrecht University\\
    Janskerkhof 13, 3512 BL Utrecht\\ 
The Netherlands\\
    \printead{e2} }
      \and%
  \author{\fnms{Fan}
    \snm{Yang}\ead[label=e3]{fan.yang.c@gmail.com}}
  \address{Department of Values, Technology and Innovation\\
    Delft University of Technology\\
    Jaffalaan 5, 2628 BX Delft\\ 
The Netherlands\\
    \printead{e3} }


  \runauthor{I. Ciardelli, R. Iemhoff and F.~Yang}

\begin{abstract}
In recent years, the logic of questions and dependencies has been investigated in the closely related frameworks of inquisitive logic and dependence logic. These investigations have assumed classical logic as the background logic of statements, and added formulas expressing questions and dependencies to this classical core. In this paper, we broaden the scope of these investigations by studying questions and dependency in the context of intuitionistic logic. We propose an intuitionistic team semantics, where teams are embedded within intuitionistic Kripke models. The associated logic is a conservative extension of intuitionistic logic with questions and dependence formulas. We establish a number of results about this logic, including a normal form result, a completeness result, and translations to classical inquisitive logic and modal dependence logic. 
\end{abstract}



\end{frontmatter}

\section{Introduction}

Traditionally, the role of a semantics for a logic is to provide a relation of truth between the formulas of the logic and some mathematical objects that stand for states of affairs. For instance, in the standard semantics for classical propositional logic, a state of affairs is modeled by a propositional valuation, and the semantics specifies when a valuation makes a formula true. The fundamental logical relation of entailment between formulas is then characterized as preservation of truth: an entailment holds when the conclusion is true whenever all the assumptions are.

The last decade, however, has seen the rise of logics in which formulas are evaluated not with respect to objects that represent state of affairs---say, propositional valuations---but with respect to \emph{sets} of such objects, that we will refer to as \emph{teams}. Such logics have arisen independently in two different lines of research, namely, \emph{dependence logic} and \emph{inquisitive logic}.

The main motivation of dependence logic is to expand systems of classical logic with formulas that express the existence of certain dependencies. For instance, consider a formula $\dep(p,q)$ which expresses the fact that the truth value of a proposition $q$ is completely determined by the truth value of another proposition $p$. In a standard semantic framework, it is not clear how such a formula should be interpreted: with respect to a specific propositional valuation, both $p$ and $q$ have well-defined truth-values, and it is not clear what it would mean for the value of $q$ to be determined by the value of $p$. This can be solved once we turn from single valuations to sets of valuations. In such a set $t$, it is very clear what it means for the value of $q$ to be determined by the value of $p$: this means that if two valuations in $t$  agree on the value of $p$, they must also agree on the value of $q$. We can take this to be the satisfaction condition for the formula $\dep(p,q)$ at $t$. In general, the fundamental idea is that dependencies manifest themselves in the presence of a plurality of states of affairs, and should therefore be viewed as properties of sets of states of affairs rather than as properties of single states of affairs.

The main motivation for inquisitive logic comes from a different limitation of the truth-conditional approach. While truth-conditional semantics provides a suitable starting point for an analysis of statements, such as \emph{it is raining}, it does not seem equally suited to analyze questions, such as `\emph{whether or not it is raining}' or `\emph{whether it is raining or snowing}'. For concreteness, consider a formula $?p$ which expresses the polar question \emph{whether or not $p$}. Again, it is not clear what it should mean for such a question to be true relative to a single propositional valuation. However, we can interpret this question naturally relative to a set $t$ of valuations: $?p$ will be supported by $t$ if all the valuations in $t$ agree on whether $p$ is the case. More generally, the idea is that a team $t$ can be viewed as an information state: the information encoded by $t$ is that the actual state of affairs is an element of $t$. The idea is then to interpret both statements and questions in terms of whether the information in $t$ suffices to support them, where supporting a statement means establishing that it is true, and supporting a question means settling the issue it expresses. In this way, team semantics provides a more general semantic framework which is suitable for both kinds of sentences.


While these two lines of research have arisen independently from one another, they are in fact deeply related. Yang \cite{Yang:14}  first noticed and exploited a tight similarity between propositional systems of inquisitive and dependence logic. Later on, Ciardelli \cite{Ciardelli:16dependency} argued that this similarity is not accidental, and that a fundamental relation exists between questions and dependency. Namely, the relation of dependency is just a facet of the fundamental logical notion of entailment, once this applies to questions, rather than statements. Given the link existing between entailment and the implication operator, this also provides a general way to express dependencies as implications between questions, generalizing the pattern expressed by the specialized atoms of dependence logic. Taking inspiration from this connection, Yang and V\"{a}\"{a}n\"{a}nen developed a simplified deduction system for propositional dependence logic in \cite{YangVaananen:16}.

In the past decade, both areas of research have grown rapidly.\footnote{For dependence logic see, e.g., \cite{Vaananen:07,Vaananen:08,GradelVaananen:13,Galliani:12,GallianiHella:13fp,EHMMVV2013,Yang2011,YangVaananen:16,YangVaananen:17PT}; for inquisitive logic see, e.g., \cite{CiardelliRoelofsen:11jpl,CiardelliRoelofsen:15idel,Ciardelli:14aiml,Ciardelli:16qait,Ciardelli:16,Ciardelli:15inqd,Puncochar:15weaknegation,Frittella:16}.} With the exception of Pun\v coch\'a\v r's  work \cite{Puncochar:15generalization,Puncochar:16algebras}, discussed in Section \ref{sec:related work}, the research has  focused on enriching existing systems of classical logic (propositional, modal, or first-order) with questions and dependence formulas. 
 However, there is no \emph{a priori} reason why investigating the logic of questions and dependency would require a commitment to an underlying classical logic of statements. A major open question that has so far received little attention  is how questions and dependencies should be analyzed semantically---and what logical features they have---in a non-classical setting.

This question is interesting not just because non-classical logics play an important role in a number of areas---from constructive mathematics to philosophy to technological applications---but also because it would allow us to understand much better which features of  inquisitive and dependence logics are due to the classicality of the underlying logic of statements, and which features depend only on the way in which questions and dependencies are related to the underlying logical basis---regardless of what this is taken to be.

In this paper we take a first step towards exploring this important question by investigating propositional questions and dependencies in the context of intuitionistic logic. The semantics that we will propose can be viewed as a generalization of the standard semantics for propositions inquisitive and dependence logic, which in turn can be recovered by specializing our semantics to a particular class of models. We will see that, once we start out with the right set-up for the classical case, the generalization works out smoothly, and many results carry over  straightforwardly from the classical to the intuitionistic setting. This includes the strong normal form result that characterizes propositional inquisitive  logic, as well as an axiomatization of the associated logic. In terms of a natural deduction system,  the only difference between a classical system and our intuitionistic system lies in the availability of the double negation law for standard propositional formulas. On the other hand, we will also see that the intuitionistic setting is in many ways more fine-grained than the classical one. For instance, even in terms of truth with respect to single worlds, standard disjunctions come apart from inquisitive disjunctions---in contrast with the situation in the classical case. Similarly, some syntactic manipulation techniques that are heavily exploited in inquisitive logic, such as applying double negation to obtain a non-inquisitive formula, are no longer available in the intuitionistic case.

The paper is structured as follows: in Section 2 we review how a team-based approach allows us to add questions and dependence formulas to classical propositional logic; we also recall a number of important notions and results that we will later consider in the intuitionistic setting. In Section 3 we provide a team-based semantics for intuitionistic propositional logic, and we show that this allows us to introduce questions and dependence formulas in the intuitionistic setting. In Section 4 we study in detail the logic that arises from this system. The relations between our proposal and recent work by Pun\v coch\'a\v r \cite{Puncochar:15generalization,Puncochar:16algebras} are discussed in Section 5. Section 6 sums up our findings and outlines some directions for further work.

\section{Background: questions in classical logic}

In this section, we provide some background on how questions and dependence formula can be added on top of classical propositional logic, and what logic results from this move. 
Following Ciardelli \cite{Ciardelli:16dependency}, we first review how classical propositional logic can be re-implemented in a team semantic setting, then show how questions can be added to this system, and how dependencies arise naturally as implications among questions. The system that we present 
is an extension of standard propositional inquisitive logic \inqb\ \citep{CiardelliRoelofsen:11jpl} with the tensor disjunction connective from dependence logic (introduced by V\"a\"an\"anen \cite{Vaananen:07}, and studied in the propositional setting by Yang \cite{YangVaananen:16}). We will refer to this system as \lqd. For proofs of the results presented in this section, the reader is referred to \cite{Ciardelli:16dependency}.


\subsection{Support semantics for classical logic}
\label{sub:support for classical logic}

The first step to get at a classical logic of questions and dependencies is to provide a semantics for classical propositional logic (\cpl) which is based not, as usual, on the notion of truth relative to a state of affairs, but rather on the notion of support relative to a state of information.

\subsubsection{Syntax}

For our purposes, it will be convenient to start from a formulation of classical logic which takes the connectives $\bot,\land,\lor,$ and $\to$ as primitive operators. Standard propositional formulas are formulas built up from a set $\P$ of atomic sentences by means of these operators. We will denote the set of standard formulas by $\L_!$, and we will think of this as our language for statements, on top of which questions will be added in the next section. 
%
%
Negation is regarded as a defined connective, by setting $\neg\phi:=\phi\to\bot$. 

\subsubsection{Models}

The context for our semantics is provided by a \emph{possible world model} for classical logic: this is a model that represents at once a multitude of possible states of affairs, called \emph{possible worlds}, each of which is completely described by a valuation function for the atomic sentences of our language.\footnote{This semantic set-up differs slightly from the one assumed in most previous work on propositional inquisitive and dependence logic \citep[e.g.,][]{CiardelliRoelofsen:11jpl,YangVaananen:16}. Most previous work assumes a fixed model $\omega$, having the propositional valuations themselves as possible worlds. Since this model contains a copy of each possible state of affairs, the difference between the two setups is immaterial to the logic. The setup used here is adopted in \cite{Ciardelli:16} as it facilitates the transition between the propositional setting and the modal setting. Similarly, in this paper, this choice facilitates the transition from a classical to an intuitionistic semantic basis. In both cases, one only needs to equip the space of possible worlds with additional structure on top of the propositional valuation. }

\begin{definition}[Classical possible world models]~\\
A \emph{classical possible world model} for a set \P\ of atoms is a pair $M=\<W,V\>$, where:
\begin{itemize}
\item  $W$ is a set, whose elements we refer to as \emph{possible worlds}; 
\item $V:W\times\P\to\{0,1\}$ is map that we refer to as the \emph{valuation function}.
\end{itemize}
\end{definition}

\noindent
Given a model $M=\<W,V\>$, we refer to a set $t\subseteq W$ as a \emph{team}. Below, we will use $t$ and $s$, as well as variants like $t',t'',\dots$ as meta-variables ranging over teams.

Intuitively, a team may be thought of as encoding a body of information: if $w\in t$, this means that $w$ is compatible with the information available in $t$; if $w\not\in t$, then $t$ is ruled out by the information in $t$. In other words, we can view $t$ as encoding the information that the actual state of affairs is one of those contained in $t$. Due to this informational interpretation, in the inquisitive semantics literature teams are referred to as \emph{information states}.
If $t'\subseteq t$, this means that $t'$ contains at least as much information as $t$, and possibly more. In this case, we say that $t'$ is at least as strong as $t$, or an \emph{extension} of~$t$. The weakest of all teams is the total team, $W$, which is compatible with all possible worlds. The strongest team is the empty team, $\emptyset$, which is compatible with no possible world. We refer to $\emptyset$ as the \emph{inconsistent team}, and to any team $t\neq\emptyset$ as a \emph{consistent team}.

\subsubsection{Semantics}

Standardly, in a possible world model a semantics is specified in the form of a relation of \emph{truth}, holding between formulas and possible worlds.  By contrast, in inquisitive and dependence logic, a semantics is given in terms of a relation of \emph{support}, holding between formulas and \emph{teams}. Intuitively, if formulas of classical logic are, as usual, regarded as statements, then the relation $t\models\phi$ may be read as specifying when the information available in $t$ suffices to establish that $\phi$ is the case. 

\begin{definition}[Support semantics for classical logic]\label{def:support classical}~\\
Let $M$ be a possible world model. The relation of support between teams $t$ and formulas $\phi\in\L_!$ is defined inductively as follows:
\begin{itemize}
\item $M,t\models p\iff V(w,p)=1$ for all $w\in t$
\item $M,t\models\bot\iff t=\emptyset$
\item $M,t\models\phi\land\psi\iff M,t\models\phi$ and $M,t\models\psi$
\item $M,t\models\phi\lor\psi\iff \exists t',t''\text{ such that }t=t'\cup t''$, $M,t'\models\phi$ and $M,t''\models\psi$ 
\item $M,t\models\phi\to\psi\iff\forall t'\subseteq t$, $\,M,t'\models\phi\,$ implies $\,M,t'\models\psi$
\end{itemize}
\end{definition}

\noindent The clauses can be read as follows: an atom $p$ is established in $t$ in case only worlds in which $p$ is true are compatible with the information in $t$. The falsum constant $\bot$ is established in $t$ just in case $t$ is inconsistent. A conjunction is established in $t$ in case both conjuncts are established. A disjunction is established in $t$ if the possible worlds compatible with $\phi$ can be subdivided in two teams $t'$ and $t''$, where $t'$ establishes $\phi$ and $t''$ establishes $\psi$. Finally, an implication $\phi\to\psi$ is established in $t$ if 
extending the information in $t$ so as to establish $\phi$ is bound to lead to a state where $\psi$ is established as well.   These clauses yield the following clause for negation: $\neg\phi$ is established in $t$ if it extending $t$ so as to establish  $\phi$ is bound to lead to inconsistency.

\begin{itemize}
\item $M,t\models\neg\phi\iff\forall t'\subseteq t$, $\;M,t'\models\phi\,$ implies $\,t'=\emptyset$
\end{itemize}

\subsubsection{Truth and conservativity}

\noindent 
Even though our semantics is based on the notion of support at a team, the standard notion of truth can be recovered: it suffices to define truth at a world as support with respect to the corresponding singleton team. 

\begin{definition}[Truth conditions]~\\
Let $M$ be a model and $w$ a world. We say that a formula $\phi$ is \emph{true} at $w$, and write $M,w\models\phi$, in case $M,\{w\}\models\phi$.
\end{definition}

\noindent The following proposition shows that this indeed coincides with the standard notion of truth in classical propositional logic. 

\begin{proposition}[Truth-conditions for standard formulas]\label{prop:truth-conditions classical}~\\
For any model $M=\<W,V\>$ and any world $w\in W$ we have: 
\begin{itemize}
\item $M,w\models p\iff V(w,p)=1$
\item $M,w\not\models\bot$
\item $M,w\models\phi\land\psi\iff M,w\models\phi$ and $M,w\models\psi$
\item $M,w\models\phi\lor\psi\iff M,w\models\phi$ or $M,w\models\psi$
\item $M,w\models\phi\to\psi\iff M,w\not\models\phi$ or $M,w\models\psi$
\end{itemize}
\end{proposition}

\noindent
This proposition shows that the standard notion of truth is completely determined by our support definition. Conversely, let us say that a formula is \emph{truth-conditional} when its support conditions are completely determined by its truth-conditions, in the sense that support at a team $t$ amounts to truth at each world in $t$.

\begin{definition}[Truth-conditionality]~\\
We say that a formula $\phi$ is \emph{truth-conditional} in case for any model $M$ and team $t$: 
$$M,t\models\phi\iff \forall w\in t:\; M,w\models\phi$$
\end{definition}

\noindent Truth-conditional formulas are also called \emph{flat} formulas in the dependence logic literature.
It is easy to show by induction that all standard propositional formulas are in fact truth-conditional. Thus, for standard formulas the relation of truth completely determines our relation of support.
\begin{proposition}\label{prop:truth-conditionality classical} All standard formulas are truth-conditional
\end{proposition}

\noindent Summing up, then, the support semantics given above is inter-definable with standard truth-conditional semantics. It is easy to see that this implies that the resulting logic---where entailment is defined as support preservation---is just classical propositional logic. Thus, what we have done so far is simply to re-implement classical propositional logic in a team semantic setting.

\subsection{Adding questions to classical propositional logic}
\label{sub:adding questions}

Now that we have given a support semantics for classical logic, the stage is set for questions to enter the picture. 

\subsubsection{Inquisitive disjunction}

We achieve this by enriching our language with a new connective $\lori$, called \emph{inquisitive disjunction}.  Thus, the full language \L\ of our logic \lqdp\ is given by the following definition.


\begin{definition}[Language \Lp]\label{def:full language}~\\
The set \L\ of formulas of \lqd\ is defined inductively as follows:%
$$\phi\;::=\;p\,|\,\bot\,|\,\phi\land\phi\,|\,\phi\lor\phi\,|\,\phi\to\phi\,|\,\phi\lori\phi$$%
\end{definition}

\noindent
The new connective $\lori$ is interpreted by means of the following support clause.
\begin{definition}[Support clause for inquisitive disjunction]~
\begin{itemize}
\item $M,t\models\phi\lori\psi\iff M,t\models\phi\,\text{ or }\,M,t\models\psi$
\end{itemize}
\end{definition}

\noindent Intuitively, we can regard a formula like $p\lori q$ as a question \emph{whether $p$ or $q$}. To support this question, the information available in the team should establish a specific one of the disjuncts. In general, when we regard a formula $\mu$ as a question, we can read the relation $t\models\mu$ as meaning that the information available in $\mu$ settles the question $\mu$. (For a more detailed discussion of how the notion of support provides a unified semantics for statements and questions, see \cite{Ciardelli:16qait}.)

Notice that if $\alpha$ is a standard formula, then the disjunction $\aa\lori\neg\aa$ captures the polar question \emph{whether or not $\alpha$}, which is supported by a team $t$ if the information in $t$ implies that $\aa$ is true, or it implies that $\aa $ is false. This observation justifies the introduction of the following abbreviation, which is standard in inquisitive semantics and which will be useful throughout the paper. 
$$?\phi:=\phi\lori\neg\phi$$

\subsubsection{Persistency and truth-conditionality}

\noindent The logical system we just defined satisfies the following two properties. 

\begin{proposition}[Persistency and empty team property]~
\begin{itemize}
\item Persistency property: if $M,t\models\phi$ and $s\subseteq t$, then $M,s\models\phi$
\item Empty team property: $M,\emptyset\models\phi$ for all $\phi\in\L$
\end{itemize}
\end{proposition}
\noindent The first claim states that the semantics is \emph{persistent}, in the sense that if a formula is supported by a team $t$, then it remains supported by any extension of $t$. The second is a semantic analogue of the usual \emph{ex falso quodlibet} principle: it states that the empty team, which represents the state of inconsistent information, supports any formula.

Unlike standard formulas, formulas which contain $\lori$ are not in general truth-conditional. For instance, the formula $?p$ is true at any world in any model, but it is only supported at teams $t$ such that the truth-value of $p$ is constant within $t$. We will refer to truth-conditional formulas as \emph{statements}, and to non-truth conditional formulas as \emph{questions} (for the conceptual motivation behind this terminology, see~\citep{Ciardelli:16qait}).

\subsubsection{Statements and questions}

\noindent Proposition \ref{prop:truth-conditionality classical} implies that for a statement, we always have a unique maximal supporting state, which coincides with the set of all worlds in which the statement is true. As an illustration, Figures \ref{fig:p}-\ref{fig:p to q} depict the maximal supporting states for some standard formulas.

By contrast, questions typically have multiple maximal supporting states, that we can think of intuitively as mirroring the alternative ways in which the question can be resolved. Figures \ref{fig:p lori q}-\ref{fig:p to ?q} depict the maximal supporting states for some questions in \lqd. Notice that while $\lori$ allows us to form questions out of statements, these questions can then be combined by means of the standard connectives. Thus, for instance, the conjunction ${?p}\land{?q}$ is a question which is only resolved when both conjuncts are resolved, while the implication $p\to{?q}$ is a question which is resolved if the question $?p$ is resolved restricted to those worlds in which $p$ is true.

\noindent 
\begin{figure}[t]
\centering
\subfigure[{$p$}]{\label{fig:p}
 \begin{tikzpicture}[>=latex,scale=0.5]

 \draw[opaque,rounded corners] (-1.8,1.8) rectangle (1.8, .2);
 \draw (1.9,1.9) rectangle (1.9,1.9); 
 \draw (-1.9,-1.9) rectangle (-1.9,-1.9); 

 \draw (-1,1) node[index gray] (yy) {$11$};
 \draw (1,1) node[index gray] (yn) {$10$};
 \draw (-1,-1) node[index gray] (ny) {$01$};
 \draw (1,-1) node[index gray] (nn) {$00$};

	\end{tikzpicture}
}
\hspace{0.04in}
\subfigure[{$\neg p$}]{\label{fig:not p}
 \begin{tikzpicture}[>=latex,scale=0.5]

 \draw[opaque,rounded corners,yshift=-2cm] (-1.8,1.8) rectangle (1.8, .2);
 \draw (1.9,1.9) rectangle (1.9,1.9); 
 \draw (-1.9,-1.9) rectangle (-1.9,-1.9); 

 \draw (-1,1) node[index gray] (yy) {$11$};
 \draw (1,1) node[index gray] (yn) {$10$};
 \draw (-1,-1) node[index gray] (ny) {$01$};
 \draw (1,-1) node[index gray] (nn) {$00$};

	\end{tikzpicture}
}
\hspace{.04in}
\subfigure[{$p\land q$}]{\label{fig:p and q}
 \begin{tikzpicture}[>=latex,scale=0.5]

 \draw[opaque,rounded corners] (-1.8,1.8) rectangle (-.2,.2);
 \draw (1.9,1.9) rectangle (1.9,1.9); 
 \draw (-1.9,-1.9) rectangle (-1.9,-1.9); 

 \draw (-1,1) node[index gray] (yy) {$11$};
 \draw (1,1) node[index gray] (yn) {$10$};
 \draw (-1,-1) node[index gray] (ny) {$01$};
 \draw (1,-1) node[index gray] (nn) {$00$};

	\end{tikzpicture}
}
\hspace{.04in}
\subfigure[{$p\lor q$}]{\label{fig:p tensor q}
 \begin{tikzpicture}[>=latex,scale=0.5]

 \draw[opaque,rounded corners,rotate=-90] (-1.8,0) -- (-1.8,1.8) -- (-.5,1.8) -- (.2,.2) -- (1.8,-.5) -- (1.8,-1.8) -- (-1.8,-1.8) -- (-1.8,0);
 \draw (1.9,1.9) rectangle (1.9,1.9); 
 \draw (-1.9,-1.9) rectangle (-1.9,-1.9); 

 \draw (-1,1) node[index gray] (yy) {$11$};
 \draw (1,1) node[index gray] (yn) {$10$};
 \draw (-1,-1) node[index gray] (ny) {$01$};
 \draw (1,-1) node[index gray] (nn) {$00$};

	\end{tikzpicture}
}
\hspace{0.04in}
\subfigure[{$p\to q$}]{\label{fig:p to q}
 \begin{tikzpicture}[>=latex,scale=0.5]

 \draw[opaque,rounded corners] (-1.8,0) -- (-1.8,1.8) -- (-.5,1.8) -- (.2,.2) -- (1.8,-.5) -- (1.8,-1.8) -- (-1.8,-1.8) -- (-1.8,0);
 \draw (1.9,1.9) rectangle (1.9,1.9); 
 \draw (-1.9,-1.9) rectangle (-1.9,-1.9); 

 \draw (-1,1) node[index gray] (yy) {$11$};
 \draw (1,1) node[index gray] (yn) {$10$};
 \draw (-1,-1) node[index gray] (ny) {$01$};
 \draw (1,-1) node[index gray] (nn) {$00$};

	\end{tikzpicture}
}
\subfigure[$p\,{\protect\disji}\,q$]{\label{fig:p lori q}
 \begin{tikzpicture}[>=latex,scale=0.5]

 \draw[opaque,rounded corners] (-1.85,1.85) rectangle (1.8, .2);
 \draw[opaque,rounded corners] (-1.75,1.75) rectangle (-.2,-1.8);
 \draw (1.9,1.9) rectangle (1.9,1.9); 
 \draw (-1.9,-1.9) rectangle (-1.9,-1.9); 

 \draw (-1,1) node[index gray] (yy) {$11$};
 \draw (1,1) node[index gray] (yn) {$10$};
 \draw (-1,-1) node[index gray] (ny) {$01$};
 \draw (1,-1) node[index gray] (nn) {$00$};

	\end{tikzpicture}
}
\subfigure[{$?p$}]{\label{fig:?p}
 \begin{tikzpicture}[>=latex,scale=0.5]

 \draw[opaque,rounded corners] (-1.8,1.8) rectangle (1.8, .2);
 \draw[opaque,rounded corners] (-1.8,-1.8) rectangle (1.8, -.2);
 \draw (1.9,1.9) rectangle (1.9,1.9); 
 \draw (-1.9,-1.9) rectangle (-1.9,-1.9); 

 \draw (-1,1) node[index gray] (yy) {$11$};
 \draw (1,1) node[index gray] (yn) {$10$};
 \draw (-1,-1) node[index gray] (ny) {$01$};
 \draw (1,-1) node[index gray] (nn) {$00$};

	\end{tikzpicture}
}
\hspace{.04in}
\subfigure[{${?q}$}]{\label{fig:?q}
 \begin{tikzpicture}[>=latex,scale=0.5]

\draw[opaque,rounded corners,rotate=90] (-1.8,1.8) rectangle (1.8, .2);
\draw[opaque,rounded corners,rotate=90] (-1.8,-.2) rectangle (1.8,-1.8);
 \draw (1.9,1.9) rectangle (1.9,1.9); 
 \draw (-1.9,-1.9) rectangle (-1.9,-1.9); 


 \draw (-1,1) node[index gray] (yy) {$11$};
 \draw (1,1) node[index gray] (yn) {$10$};
 \draw (-1,-1) node[index gray] (ny) {$01$};
 \draw (1,-1) node[index gray] (nn) {$00$};

	\end{tikzpicture}
}
\hspace{.04in}
\subfigure[{${?p}\wedge{?q}$}]{\label{fig:?p and ?q}
 \begin{tikzpicture}[>=latex,scale=0.5]

 \draw[opaque,rounded corners] (-1.8,1.8) rectangle (-.2,.2);
 \draw[opaque,rounded corners] (-1.8,-1.8) rectangle (-.2,-.2);
 \draw[opaque,rounded corners] (1.8,1.8) rectangle (.2,.2);
 \draw[opaque,rounded corners] (1.8,-1.8) rectangle (.2,-.2);
 \draw (1.9,1.9) rectangle (1.9,1.9); 
 \draw (-1.9,-1.9) rectangle (-1.9,-1.9); 

 \draw (-1,1) node[index gray] (yy) {$11$};
 \draw (1,1) node[index gray] (yn) {$10$};
 \draw (-1,-1) node[index gray] (ny) {$01$};
 \draw (1,-1) node[index gray] (nn) {$00$};

	\end{tikzpicture}
}
\hspace{.04in}
\subfigure[{$p\to{?q}$}]{\label{fig:p to ?q}
 \begin{tikzpicture}[>=latex,scale=0.5]

 \draw[opaque,rounded corners] (-1.7,0) -- (-1.7,1.8) -- (-.5,1.8) -- (.2,.2) -- (1.8,-.5) -- (1.8,-1.8) -- (-1.7,-1.8) -- (-1.7,0);
 \draw[opaque,rounded corners] (1.7,0) -- (1.7,1.8) -- (.5,1.8) -- (-.2,.2) -- (-1.8,-.5) -- (-1.8,-1.8) -- (1.7,-1.8) -- (1.7,0);
 \draw (1.9,1.9) rectangle (1.9,1.9); 
 \draw (-1.9,-1.9) rectangle (-1.9,-1.9); 

 \draw (-1,1) node[index gray] (yy) {$11$};
 \draw (1,1) node[index gray] (yn) {$10$};
 \draw (-1,-1) node[index gray] (ny) {$01$};
 \draw (1,-1) node[index gray] (nn) {$00$};

	\end{tikzpicture}
}

\caption{The maximal supporting states for some propositional formulas. In the figures, 11 stands for a world in which $p$ and $q$ are true, 10 for a world in which $p$ is true and $q$ is false, and so on.}
\label{fig:meanings classical}
\end{figure}
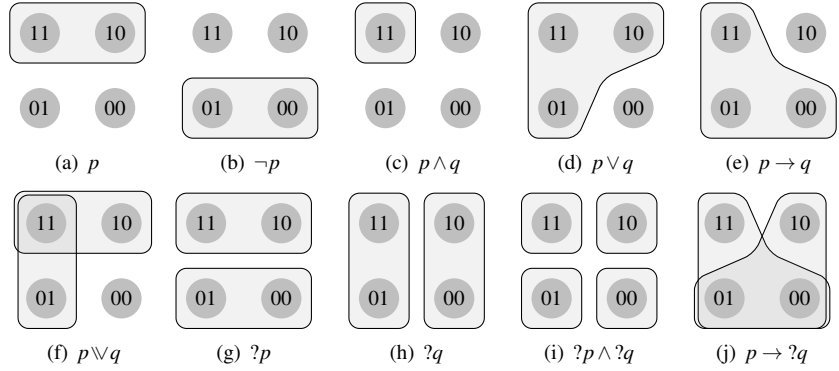

\subsubsection{Dependence formulas}
\label{sub:dependence}

\noindent Let us now turn our attention to the relation of dependency. Consider a team $t$ and consider two questions $\mu$ and $\nu$. It is natural to say that $\nu$ is \emph{determined by} $\mu$ in $t$ if, relative to the team, resolving $\mu$ implies resolving $\nu$; or, more formally, in case in any extension of $t$ where $\mu$ is resolved, so is $\nu$.
$$\mu\text{ determines }\nu\text{ in }t\iff \forall u\subseteq t: u\models\mu\text{ implies }u\models\nu$$
Notice that the condition on the right is just the support condition for $\mu\to\nu$. Thus, we have:
$$\mu\text{ determines }\nu\text{ in }t\iff t\models\mu\to\nu$$
This shows that for any two questions $\mu$ and $\nu$, we can regard the implication $\mu\to\nu$ as expressing the existence of a certain dependency relation: a team $t$ supports $\mu\to\nu$ if and only if relative to $t$, the antecedent completely determines the consequent. For this reason, we will refer to a formula of the form $\mu\to\nu$ as a \emph{dependence formula}.

For a simple example, consider the implication ${?p}\to{?q}$. This formula is supported in a team $t$ in case relative to $t$, whether $q$ is true is completely determined by whether $p$ is true. This happens if and only if within $t$ we have a functional dependency of the value of $q$ on the value of $p$:
$$M,t\models{?p}\to{?q}\iff\exists f:\{0,1\}\to\{0,1\}\text{ s.t.\ }\forall w\in t:V(w,q)=f(V(w,p))$$
Another way of stating this condition is the following:
$$M,t\models{?p}\to{?q}\iff\forall w,w'\in t:V(w,p)=V(w',p)\text{ implies }V(w,q)=V(w',q)$$
This is precisely the support condition which is assumed for the dependence atom $\dep(p,q)$ in propositional dependence logic \citep{Vaananen:08,YangVaananen:16}. Thus, in our logic we can take $\dep(p,q)$ to be an abbreviation for ${?p}\to{?q}$.

The point we just made generalizes to the case in which a question is jointly determined by a number of other questions. If $\mu_1,\dots,\mu_n,\nu$ are questions, the implication $\mu_1\land\dots\land\mu_n\to\nu$ expresses the fact that $\nu$ is jointly determined by $\mu_1,\dots,\mu_n$, that is, the fact that in the given team, $\nu$ is resolved as soon as each of $\mu_1,\dots,\mu_n$ is resolved.

In particular, if $\aa_1,\dots,\aa_n,\bb$ are standard formulas, then we can express the fact that the truth-value of $\bb$ is completely determined by those of $\aa_1,\dots,\aa_n$  by means of the formula ${?\aa_1}\land\dots\land{?\aa_n}\to {?\bb}$. We can thus take the generalized dependence atom $\dep(\aa_1,\dots,\aa_n,\bb)$ from dependence logic as an abbreviation of this formula.

At the same time, however, in \lqd\ we can express dependence patterns that do not correspond to dependence atoms (for discussion of this point, see \citep{Ciardelli:16dependency}). For instance, the formula $?p\to q\lori r$ expresses the fact that in the given team, settling the truth-value of $p$ leads to establishing one of $q$ and $r$.

%

\subsection{Resolutions and normal form}
\label{sub:normal form classical}

In Section \ref{sub:support for classical logic} we saw that any standard formula is truth-conditional. Conversely, we can show that any truth-conditional formula in $\L$ is equivalent to a standard formula. This means that, while \lqdp\ is more expressive than classical logic, it has the same expressive power as classical logic as far as statements are concerned.\footnote{This is not the case for stronger logics. In inquisitive modal logic \citep{CiardelliRoelofsen:15idel,Ciardelli:14aiml,Ciardelli:16}, the presence of questions embedded under modalities allows us to express statements that have no counterpart in a standard modal language.} To show this, we associate with any formula $\phi$ a corresponding standard formula $\phi^s$ which has the same truth-conditions.

\begin{definition}[Standard variant]~\\
The standard variant of a formula $\phi\in\L$ is the formula $\phi^s\in\L_!$ obtained by replacing each occurrence of $\lori$ by $\lor$.
\end{definition}

\begin{proposition}\label{prop:standard variant} For any $\phi\in\L$, $\phi$ and $\phi^s$ have the same truth-conditions.
\end{proposition}

\noindent If $\phi$ is itself truth-conditional, then this guarantees that $\phi$ and $\phi^s$ are equivalent. Clearly, the opposite holds as well.

\begin{proposition} $\phi$ is truth-conditional $\iff\phi\equiv\phi^s$
\end{proposition}

\noindent
Thus, if $\phi$ is truth-conditional, then $\phi$ is equivalent with a standard formula.  
If $\phi$ is a question, then obviously $\phi$ is not equivalent to any standard formula. Nevertheless, any formula of \lqd\ is equivalent to an \emph{inquisitive disjunction} of standard formulas. This means that the inquisitiveness which is present in a formula can always be brought all the way to the surface syntactic layer. The first step for this is to associate with any formula $\phi$ a finite set of standard formulas, called \emph{resolutions} of $\phi$.

\begin{definition}[Resolutions]\label{def:resolutions}~\\
The set $\R(\phi)$ of resolutions of a formula $\phi$ is defined inductively as follows.
\begin{itemize}
\item $\R(p)=\{p\}$
\item $\R(\bot)=\{\bot\}$
\item $\R(\phi\land\psi)=\{\aa\land\bb\,|\,\aa\in\R(\phi)\text{ and }\bb\in\R(\psi)\}$
\item $\R(\phi\lor\psi)=\{\aa\lor\bb\,|\,\aa\in\R(\phi)\text{ and }\bb\in\R(\psi)\}$
\item $\R(\phi\to\psi)=\{\bigwedge_{\aa\in\R(\phi)}\aa\to f(\aa)\,|\,f:\R(\phi)\to\R(\psi)\}$
\item $\R(\phi\lori\psi)=\R(\phi)\cup\R(\psi)$
\end{itemize}
\end{definition}

\noindent
It is then possible to prove by induction that any formula is equivalent with the inquisitive disjunction of its resolutions.

\begin{theorem}[Inquisitive normal form]\label{theor:normal form cl}~\\
If $\phi\in\L$ and $\R(\phi)=\{\aa_1,\dots,\aa_n\}$ we have $\phi\equiv\aa_1\lori\dots\lori\aa_n$.
\end{theorem}

\noindent
This is a very strong normal form result, from which many features of the logic can be deduced. 
It turns out that this result is quite stable with respect to the logic of statements that we choose as a starting point for our construction: in Section \ref{sub:normal form} we will prove that exactly the same normal form result holds in the intuitionistic setting.

\subsection{The logic}
\label{sub:logic classical}

Standardly, semantics is given in terms of truth, and logical entailment is characterized as truth preservation. In team-based logics, semantics is given in terms of support, and logical entailment is defined as support preservation.

\begin{definition}[Entailment]~\\
$\Phi\models_{\lqdp}\psi\iff\text{for any }M,t,\text{ if }M,t\models\phi\text{ for all }\phi\in\Phi,\text{ then }{M,t\models\psi}$
\end{definition}

\noindent 
As mentioned already in Section \ref{sub:support for classical logic}, the system \lqd\ is a conservative extension of classical propositional logic.

\begin{proposition}[Conservativity over classical logic]~\\
If $\Phi\cup\{\psi\}\subseteq\L_!$, then $\Phi\models_{\lqdp}\psi\iff \Phi\models_{\textsf{CPL}}\psi$.
\end{proposition}

\noindent Essentially, this says that we have preserved our classical logic of statements. However, now it is not only statements, but also questions that can take part in the relation of entailment. This allows us to capture as cases of entailment some interesting logical notions besides the standard consequence relation between statements. For instance, as discussed in \cite{Ciardelli:16qait,Ciardelli:16}, an entailment $\aa\models\mu$, where $\aa$ is a statement and $\mu$ is a question, holds if and only if $\aa$ provides sufficient information to resolve $\mu$. Thus, for instance, we have $p\models {?(p\lor q)}$, but $\neg p\not\models {?(p\lor q)}$.

Even more interestingly, an entailment $\Gamma,\Lambda\models\mu$, where $\Gamma$ is a set of statements, $\Lambda$ is a set of questions, and $\mu$ is a question, holds if and only if given the statements in $\Gamma$, the question $\mu$ is completely determined by the questions in $\Lambda$. Thus, for instance, the entailment 
$p\leftrightarrow q,\;{?p}\,\models\, {?q}$ 
 captures the fact that, given the assumption $p\leftrightarrow q$, the question $?q$ is completely determined by the question $?p$.

Having briefly discussed the significance of entailment in \lqd, let us now turn to the formal properties of this relation. First, classical logical laws do not generally extend from statements to questions. In particular, the double negation law is valid for statements, but not for questions. 

\begin{proposition}[Double negation and truth-conditionality]\label{prop:double negation}~\\
For any formula $\phi\in\L$, $\phi\equiv\neg\neg\phi\iff \phi$ is truth-conditional.
\end{proposition}

\noindent This result shows, in particular, that the set of valid formulas in \lqd\ is not closed under uniform substitution: while $\neg\neg p\to p$ is valid in \lqd, because $p$ is truth-conditional, substituting $p$ with a question, say with $?p$, yields an invalid formula. Conceptually, the reason why uniform substitution fails in \lqd\ is that, in this system, atoms are not supposed to stand for arbitrary sentences. Rather, they are supposed to stand for arbitrary \emph{statements}; as such, they may validate logical principles which are generally valid for statements, but which fail for questions. It is important to understand that this feature of \lqd\ is not an accident, but a deliberate architectural choice (for a discussion of the benefits of this choice, see \S2.5.5 of \cite{Ciardelli:16}).

The fact that the double negation law fails for questions also shows that questions have constructive logical features. Another indication of this comes from the following proposition, which establishes a generalized version of the disjunction property for $\lori$. Intuitively, what this proposition says is that if a set of statements logically resolves a question, then it must entail a specific answer to it.

\begin{proposition}[Split property]\label{prop:split classical}~\\
If $\Gamma$ is a set of truth-conditional formulas, $\Gamma\models\phi\lori\psi$ implies $\Gamma\models\phi$ or $\Gamma\models\psi$.
\end{proposition}

\noindent
In fact, a language-internal version of this property holds as well.

\begin{proposition}[Internal split property]\label{prop:internal split classical}~\\
If $\alpha$ is a truth-conditional formula, then $\alpha\to(\phi\lori\psi)\models(\aa\to\phi)\lori(\aa\to\psi)$
\end{proposition}

\noindent
This validity plays an important role in the axiomatization of the logic \lqd. Figure \ref{fig:proof system lqd} shows a natural deduction system for \lqd. This system was proved to be sound and complete in \cite{Ciardelli:16dependency}, merging results obtained in \cite{CiardelliRoelofsen:11jpl} and in \cite{Yang:14}.

\begin{figure}
\begin{framed}
\begin{center}
\begin{tabular}{l r}
$\infer[(\land \mathsf{i})]{\phi\land\psi}{\phi & \psi}$ & $\qquad\infer[(\land \mathsf{e}_1)]{\phantom{\psi}\phi\phantom{\psi}}{\phi\land\psi}\quad\infer[(\land \mathsf{e}_2)]{\psi}{\phi\land\psi}$\\\\\\
\infer[(\to\!\mathsf{i})]{\phi\to\psi}{\deduce{\psi}{\deduce{\vdots}{[\phi]}}} & $\infer[(\to\!\mathsf{e})]{\psi}{\phi & \phi\to\psi}$ \\\\\\
$\infer[(\lori \mathsf{i}_1)]{\phi\lori\psi}{\phi}\quad\infer[(\lori \mathsf{i}_2)]{\phi\lori\psi}{\psi}$ & 
$\infer[(\lori \mathsf{e})]{\phantom{\psi}\chi\phantom{\psi}}{\deduce{\chi}{\deduce{\vdots}{[\phi]}} & \deduce{\chi}{\deduce{\vdots}{[\psi]}} & \phi\lori\psi}$\\\\\\
$\infer[(\lor \mathsf{i}_1)]{\phi\lor\psi}{\phi}\quad\infer[(\lor \mathsf{i}_2)]{\phi\lor\psi}{\psi}\quad$ &
$\quad\infer[(\lor \mathsf{e})]{\phantom{\psi}\aa\phantom{\psi}}{\deduce{\aa}{\deduce{\vdots}{[\phi]}} & \deduce{\aa}{\deduce{\vdots}{[\psi]}} & \phi\lor\psi}$\\\\\\
$\infer[(\lor \small\textsf{a})]{(\phi\lor\psi)\lor\chi}{\phi\lor(\psi\lor\chi)}$ & $\infer[(\lor \mathsf{r})]{\phi'\lor\psi'}{\deduce{\phi'}{\deduce{\vdots}{[\phi]}} & \deduce{\psi'}{\deduce{\vdots}{[\psi]}} & \phi\lor\psi}$\\\\\\
$\infer[(\lor \small\textsf{c})]{\psi\lor\phi}{\phi\lor\psi}$   & $\infer[(\lor\small\textsf{d})]{(\phi\lor\psi)\lori(\phi\lor\chi)}{\phi\lor(\psi\lori\chi)}$\\\\\\
$\infer[(\bot \mathsf{e})]{\phi}{\bot}$ & $\infer[(\s)]{(\aa\to\phi)\lori(\aa\to\psi)}{\aa\to(\phi\lori\psi)}$\\\\\\
$\infer[(\dn)]{\aa}{\neg\neg\aa}$
\end{tabular}
\end{center}
\end{framed}
\caption{A natural deduction system for \lqd. In these rules, $\phi,\psi,$ and $\chi$ range over arbitrary formulas, while  $\aa$ ranges only over standard formulas.}
\label{fig:proof system lqd}
\end{figure}

Some comments on this system. First, notice that conjunction, implication, and the falsum constant are governed by the standard logical rules, regardless of whether they apply to statements or to questions. In particular, this tells us that we can reason with dependence formulas just as we normally reason with implications in logic; for instance, we can try to prove that a dependency $\mu\to\nu$ holds by assuming $\mu$ and trying to derive $\nu$.  Second, notice that inquisitive disjunction is governed by the standard rules for disjunction in a natural deduction setup. Clearly, this means that standard disjunction cannot be governed by the same rules---otherwise the two disjunctions could be generally substituted for each other. Indeed, while the standard introduction rules are still valid for $\lor$, the standard elimination rule has to be restricted to conclusions which are standard formulas. Without this constraint, the rule would not be sound: for instance, it would incorrectly allow us to derive $?p$ from the tautology $p\lor\neg p$, which is obviously unsound. However, since the elimination rule for $\lor$ only allows us to infer standard formulas, we need additional rules to be able to deduce those consequences of a disjunction $\phi\lor\psi$ which are not classical formulas. This is the role of the rules $(\lor\textsf{a}),(\lor\textsf{c}),(\lor\textsf{d}),$ which stipulate that $\lor$ is associative, commutative, and distributive over $\lori$, and of the rule $(\lor\textsf{r})$, which allows us to replace each disjunct by a consequence of it. The system also contains the rule \s, which encodes the internal split property of Proposition \ref{prop:internal split classical}.\footnote{Early work on inquisitive logic \citep{Ciardelli:09thesis,CiardelliRoelofsen:11jpl}, did not use the split scheme \textsf{S}. Rather, the split property was captured by the well-known Kreisel-Putnam scheme
$$\textsf{KP}:=(\neg\phi\to\psi\lori\chi)\to(\neg\phi\to\psi)\lori(\neg\phi\to\chi)$$
In other words, the semantic restriction to truth-conditional formulas in Proposition \ref{prop:internal split classical} was matched in the proof system by a syntactic restriction to negative formulas, whereas in the present system it is matched by a restriction to standard formulas. In the classical case, both formulations work, essentially because any truth-conditional formula is equivalent both to a standard formula and to a negation. However, for our purposes it is crucial to use \text{S}, and not \text{KP}: this is because, in the intuitionistic setting that we are going to explore, standard formulas are still representative of all truth-conditional formulas, but negations are not. Thus, the \textsf{KP} scheme, while still valid, would only capture some particular cases of the Split Property, and would not be sufficient to obtain a complete system. This illustrates a more general point: if the classical results that we are reviewing here  carry over smoothly to the intuitionistic case, this is partly due to a careful choice between multiple classically equivalent formulations of the relevant results.} Finally, the system allows us to eliminate double negations in front of standard formulas. This encodes the fact that the logic of our standard fragment is classical.\footnote{Alternatively, we could have allowed as an axiom the law of excluded middle (\textsf{LEM}), $\aa\lor\neg\aa$, for all standard formulas $\alpha$. It is well-known that adding either \textsf{LEM} or \textsf{DNE} to \ipl\  yields classical logic, and this suffices to obtain completeness for \lqd.}

\section{Questions in intuitionistic logic}
\label{sec:system}

\noindent In the previous section, we have reviewed how classical propositional logic can be  enriched with questions and with formulas expressing dependencies. 
In this section we come to the novel enterprise of this paper. We will show that a similar result can be achieved when our starting point is intuitionistic logic, rather than classical logic.\footnote{In fact, our result will be slightly stronger: we will be able to add questions and dependence formulas on top of any logic which can be characterized as the logic of a certain a class of intuitionistic Kripke models.} As in the case of classical logic, we will proceed in steps. First, in Section \ref{sub:intuitionistic support} we will provide a semantics for intuitionistic propositional logic (\ipl) which is based on the notion of support at a team, rather than on the notion of truth at a possible world. Second, in Section \ref{sub:adding questions} we will enrich \ipl\ with questions by introducing inquisitive disjunction into the picture. We will call the resulting system \ilqd, where the last \textsf{I} stands for \emph{intuitionistic}. Some important features of this system, and some interesting differences with the classical case, are discussed in Section \ref{sub:semantic features}. 
Finally, in Section \ref{sub:dependency} we show that dependence relations can be expressed in \ilqd\ as implications among questions. In the next section we will then turn to the investigation of the logic that arises from this system.

\subsection{Support semantics for intuitionistic logic}
\label{sub:intuitionistic support}

As before, our starting point is the language $\L_!$ of standard propositional logic, consisting of formulas built up from atoms and $\bot$ by means of the connectives $\land,\lor,$ and $\to$. We want to provide a semantics for this language based on the notion of support at a team. However, now we want this semantics to characterize intuitionistic logic, rather than classical logic.  For this purpose, we need to equip our space of possible world with some extra structure besides a propositional valuation. This leads us naturally to the standard notion of an intuitionistic Kripke model.

\begin{definition}[Intuitionistic Kripke models]~\\
An intuitionistic Kripke model is a triple $M=\<W,R,V\>$, where:
\begin{itemize}
\item $W$ is a set, whose elements we will call \emph{possible worlds};
\item $R\subseteq W\times W$ is a partial order, i.e., a reflexive, antisymmetric and transitive relation;
\item $V: W\times\P\to \{0,1\}$ is a function which is obeys the following condition:
\begin{description}
\item[Persistency:] if $wRw'$, then $V(w,p)=1$ implies $V(w',p)=1$.
\end{description}
\end{itemize}
If $w\in W$, we let $R[w]=\{w'\,|\,wRw'\}$. If $t\subseteq W$, we let $R[t]:=\bigcup_{w\in t}R[w]$.
\end{definition}

\noindent
In an intuitionistic Kripke model, worlds stand for partial, rather that total, states of affairs. We may think of $wRw'$ as meaning that $w'$ is a refinement of $w'$, meaning that any aspect of reality which is determined at $w$ is also determined in the same way at $w'$, and possibly some more things are determined at $w'$. This justifies the persistency condition: if $p$ is determinately true at $w$, it must remain so at any refinement of $w$. Notice that this implies that $V(w,p)=0$ should not be read as ``$p$ is determinately false at $w$'', but rather as ``$p$ fails to be determinately true at $w$''. We will come back to this point later in this section, where we will define a partialized notion of the truth-value of a formula relative to a world.


As before, a \emph{team} will simply be a set of possible worlds in our model. Moreover, like before, we will think of a team as encoding a set of information: the team $t$ stands for the information that the actual state of affairs lies within~$t$. 

\begin{definition}[Teams]~\\
A \emph{team} in a Kripke model  $M=\<W,R,V\>$ is a set of worlds $t\subseteq W$. 
\end{definition}

\noindent
However, now there are two different ways in which a team $t$ may be extended: on the one hand, one may obtain more information about which state of affairs is actual, which results in some worlds being eliminated from $t$. On the other hand, a state of affairs in $t$ may become more defined, i.e., it may be replaced by one or more $R$-successors. This gives the following definition of extensions.\footnote{With this definition, the extension relation between states is always a pre-order, but not necessarily a partial order, since two distinct states may be extensions of each other. While this feature is unproblematic semantically, one may find it somewhat counterintuitive. This can be avoided if we allow only  \emph{upwards closed} sets of worlds in the semantics. Relative to such sets, the extension relation amounts to inclusion, just as in the classical case; in particular, it is a pre-order. We will discuss this different setup in Section \ref{sub:alternative setups},  and show that it is essentially equivalent to the more liberal semantics that we develop here.}

\begin{definition}[Extensions of a team]~\\ 
Let $M=\<W,R,V\>$ be a Kripke model. A team~$t'$ is an extension of a team~$t$ if $t'\subseteq R[t]$. 
\end{definition}

%

\noindent With these notions in place, we are now ready to define the intuitionistic notion of support with respect to a team in a Kripke model.


\begin{definition}[Support semantics for intuitionistic logic]~\\
 \label{defsupportsemantics}
Let $M=\<W,R,V\>$ be a Kripke model. The relation of support between teams $t$ and formulas $\phi\in\L_!$ is defined inductively as follows:
\begin{itemize}
\item $M,t\models p\iff V(w,p)=1$ for all $w\in t$
\item $M,t\models\bot\iff t=\emptyset$
\item $M,t\models\phi\land\psi\iff M,t\models\phi$ and $M,t\models\psi$
\item $M,t\models\phi\lor\psi\iff \exists t',t''\text{ such that }t=t'\cup t''$,~$M,t'\models\phi$ and $M,t''\models\psi$
\item $M,t\models\phi\to\psi\iff\forall t'\subseteq R[t]$, $\,M,t'\models\phi$ implies $M,t'\models\psi$ 
\end{itemize}
\end{definition}
\noindent
Notice that the support clauses are exactly the same as for classical logic; only the underlying notion of extensions of a team is different: whereas in classical logic we only look at teams $t'\subseteq t$, here we look at all teams $t'\subseteq R[t]$. 

This observation can be sharpened by noticing that the classical support semantics given above can be seen as a special case of the intuitionistic support semantics we just defined: possible world models for classical logic can be identified with Kripke models $M=\<W,R,V\>$ where worlds are already complete, and no world can be a proper refinement of another, i.e., the refinement relation $R$ is the identity. Let us call these Kripke models \emph{classical}. In a classical Kripke model, $R[t]=t$, and so the conditions $t'\subseteq t$ and $t'\subseteq R[t]$ coincide. Therefore, classical support semantics coincides with intuitionistic support semantics over the class of classical models.


\begin{example}
 \label{exintteammodel}
Consider the Kripke model $M$ depicted in Figure~\ref{fig:intteammodel}. As usual, we draw the relation $R$ by means of edges going upwards, and we omit edges which are implied by transitivity and reflexivity. Thus, in $M$ we have $w_1Rv$ for all worlds $v$, while $w_2Rv\iff v=w_2$. We depict the valuation $V$ by listing besides each world the atomic sentences which are true at that world; thus, in $M$ the sentence $p$ is true only at $w_4$, while $q$ is true only at $w_5$. The rectangles in the picture represent three teams $t_1,t_2,t_3$ in this model. Since $R[t_1]=\{w_2,w_3,w_4,w_5\}$, both $t_2=\{w_2,w_4\}$ and $t_3=\{w_5\}$ are extensions of $t_1$. The team $t_3$ is a maximal consistent team, in the sense that its only proper extension is the inconsistent team $\emptyset$. We have $t_3\models q$, since the only world in $t_3$ makes $q$ true, and $t_3\models\neg p$, since $t_3$ cannot be extended to a consistent team that supports $p$. In general, a maximal consistent team like $t_3$ gives rise to a complete theory, in the sense that for any formula $\phi$, either $\phi$ or $\neg\phi$ is supported. By contrast, the team $t_2$ has both $\{w_2\}$ and $\{w_4\}$ as proper consistent extensions. In this team, neither $p$ nor $\neg p$ is supported: for although $p$ is not true at all worlds in $t_2$, $t_2$ can be extended to a team where $p$ is supported, namely, $\{w_4\}$. Notice however that $t_2\models p\lor\neg p$ is supported, since $t_2=\{w_4\}\cup\{w_2\}$, and we have $\{w_4\}\models p$ and $\{w_2\}\models\neg p$. Indeed, one can show that in any team $t$ consisting only of terminal points---i.e., in every team consisting of complete worlds---we have $t\models\alpha\lor\neg\alpha$ for every standard formula $\alpha$, even though $t$ need not support either of the disjuncts. In this respect, standard disjunction differs from the inquisitive disjunction to be introduced below, which requires either disjunct to be supported in the team. 
Finally, notice that although $t_1$ does not contain any $p$-world, $t_1\not\models\neg p$, since $t_1$ can be extended to the team $\{w_4\}$ which supports $p$. In addition, now we also have $t_1\not\models p\lor\neg p$: although $t_1$ can be subdivided into $\{w_2\}$ and $\{w_3\}$, this does not help, since $\{w_3\}$ itself supports neither $p$ nor $\neg p$. This also illustrates the fact that, unlike in \lqd, in \ilqd\ singleton teams are not typically complete.

\begin{figure}
 \centering
 \begin{tikzpicture}[scale=0.7]
  \node[scale=0.8,shape=circle,draw=black] (A) at (0,0) {$w_1$};
  \node[scale=0.8,shape=circle,draw=black] (B) at (2,2) {$w_3$};
  \node[scale=0.8,shape=circle,draw=black] (C) at (4,4) {$w_5$};
  \node[scale=0.8,shape=circle,draw=black] (D) at (0,4) {$w_4$};
  \node[scale=0.8,shape=circle,draw=black] (E) at (-2,2) {$w_2$};
  \path (A) edge (B);
  \path (B) edge (C);
  \path (B) edge (D);
  \path (A) edge (E);
  \node (I) at (.6,4)   {$p$};
  \node (J) at (4.6,4)  {$q$};
  \node at (3.1,2) {$t_1$};
  \node at (5.4,4) {$t_3$};
  \node at (-.8,4.7) {$t_2$};
  \draw[opaque,rounded corners] (-2.7,2.7) rectangle (2.7,1.3);
  \draw[opaque,rounded corners] (3.25,4.6) rectangle (5,3.4); 
  \draw[opaque,rounded corners,rotate=45] (-.8,2.1) rectangle (4,3.5);
 \end{tikzpicture}
 \caption{A Kripke model $M$ in which three teams are indicated.}
 \label{fig:intteammodel}
\end{figure}
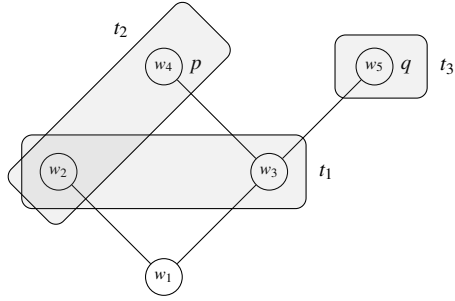
\end{example}

\noindent
From our notion of support with respect to a team we can obtain a notion of truth with respect to a possible world, in the same way as we did in the classical setting.

\begin{definition}[Truth]~\\
We say that a formula $\phi$ is true at a world $w$ in a Kripke model $M$, notation $M,w\models\phi$, in case $\phi$ is supported at the singleton state $\{w\}$:  $M,w\models\phi\iff M,\{w\}\models\phi$
\end{definition}

\noindent It will also be useful to define a notion of the truth value of a sentence with respect to a world. As we mentioned above, in the intuitionistic case, we think of possible worlds in a Kripke model as capturing partial states of affairs. This means that there will be possible worlds in which a formula is neither true nor false.

\begin{definition}[Truth-value of a formula at a world]~\\
Let $M=\<W,R,V\>$ be a Kripke model. The truth-value function associated to $M$ is the partial function $T:W\times\L_!\to\{0,1\}$ defined as follows:

$$T(w,\phi)=\left\{ 
\begin{array}{l l }
1 & \text{if }M,w\models\phi\\
0 &\text{if }M,w\models\neg\phi\\
\text{undefined} & \text{otherwise}
\end{array} \right.$$
\end{definition}

\noindent
This partial function is well-defined, since it is easy to see that $M,w\models\phi$ and $M,w\models\neg\phi$ cannot both be true. Moreover, the next proposition states that truth-values are persistent in the relation $R$.

\begin{proposition}[Persistency of truth-values]~\\
Let $M$ be a Kripke model and let $w,w'$ be two worlds in $M$ with $wRw'$. If $T(w,\phi)$ is defined, then $T(w',\phi)$ is also defined and $T(w',\phi)=T(w,\phi)$.
\end{proposition}

\begin{proof} This can be proved directly by induction, but it also follows as a special case of a more general result, namely, Proposition \ref{prop:persistency} below.
\end{proof}

\noindent
Thus, even though the basic semantic notion in our system is that of support at a team, a notion of truth at a world can be recovered straightforwardly. As in the classical case, we say that a formula is truth-conditional if support can in turn be recovered from truth, in the sense that support at a team boils down to truth at each world in the team.

\begin{definition}[Truth-conditionality]\label{truth_cond_def}~\\
We say that a formula $\phi$ is truth-conditional if for all models $M$ and teams~$t$: $M,t\models\phi\iff M,w\models\phi\text{ for all }w\in t$.
\end{definition}

\noindent
The following proposition states that all standard formulas are indeed truth-conditional, just as in the classical case.


\begin{proposition}[Standard formulas are truth-conditional]\label{prop:truth-conditionality}~\\
For any Kripke model $M$, any team $t$, and any standard formula $\phi\in\L_!$:
$$M,t\models\phi\iff M,w\models\phi\text{ for all }w\in t$$
\end{proposition}

\begin{proof} The proposition is proved by induction on the complexity of $\phi$. We only give the two non-trivial induction steps for disjunction and implication.

Case $\phi=\psi\lor\chi$. First notice that the induction hypothesis implies that $M,\emptyset\models\psi$ and $M,\emptyset\models\chi$, since the right-hand side of the equivalence is trivially satisfied for $\emptyset$. It follows easily that $M,w\models\psi\lor\chi\iff M,w\models\psi$ or $M,w\models\chi$. So, we have:
\begin{align*}
M,t\models\psi\lor\chi\iff& \exists t',t''\text{ s.t. }t=t'\cup t'',~M,t'\models\psi\text{ and }M,t''\models\chi\\
\iff& \exists t',t''\text{ s.t. }t=t'\cup t'',~(M,w'\models\psi\text{ for all }w'\in t')\text{ and }\\
&(M,w''\models\chi\text{ for all }w''\in t'')\quad\text{(by induction hypothesis)}\\
\iff&\text{for all }w\in t,\text{ either }M,w\models\psi\text{ or }M,w\models\chi\\
\iff&\text{for all }w\in t, \;M,w\models\psi\lor\chi
\end{align*}

Case $\phi=\psi\to\chi$. Suppose $M,t\models\psi\to\chi$. For any $w\in t$, we show that $M,w\models\psi\to\chi$. Let $t'\subseteq R[w]$ be such that $M,t'\models\psi$. Since $t'\subseteq R[w]\subseteq R[t]$, we obtain by the assumption that $M,t'\models\chi$, as required.

Conversely, suppose that $M,w\models\psi\to\chi$ holds for all $w\in t$. For any $t'\subseteq R[t]$ such that $M,t'\models\psi$, we show that $M,t'\models\chi$. By induction hypothesis, we have $M,w'\models\psi$ for all $w'\in t'$. Now take any $w'\in t'$: since $t'\subseteq R[t]$, we have some $w\in t$ with $wRw'$. Since $\{w'\}\subseteq R[w]$, our assumption implies $M,w'\models\chi$. Since this holds for any $w'\in t'$ by the induction hypothesis yields $M,t'\models\chi$. So, $M,t\models\psi\to\chi$.
\end{proof}

\noindent
Using this fact, we can see that the notion of truth that our semantics determines for standard formulas is nothing but the notion which is familiar from intuitionistic Kripke semantics. 



\begin{proposition}[Kripke semantics recovered]\label{prop:intuitionistic truth-conditions}~\\ 
For any Kripke model $M$, world $w$, and standard formulas $\phi$ and $\psi$, we have:
\begin{itemize}
\item $M,w\models p\iff V(w,p)=1$
\item $M,w\not\models\bot$
\item $M,w\models\phi\land\psi\iff M,w\models\phi\text{ and }M,w\models\psi$
\item $M,w\models\phi\lor\psi\iff M,w\models\phi\text{ or }M,w\models\psi$
\item $M,w\models\phi\to\psi\iff\forall v\in R[w]$, $\,M,v\models\phi$ implies $M,v\models\psi$
\end{itemize}
\end{proposition}

\begin{proof} We only give the proof for implication, which is the only case that does not follow immediately from the definition.

Suppose $M,w\models\psi\to\chi$. For any $v\in R[w]$ such that $M,v\models\psi$, since $\{v\}\subseteq R[w]$, the assumption implies $M,v\models\chi$, as required.

Conversely, suppose that for any $v\in R[w]$, $M,v\models\psi$ implies $M,v\models\chi$. Let $t\subseteq R[w]$ be such that $M,t\models\psi$. By Proposition \ref{prop:truth-conditionality}, we have $M,u\models\psi$ for all $u\in t\subseteq R[w]$. It then follows from the assumption that $M,u\models\chi$. Hence, by applying Proposition \ref{prop:truth-conditionality} again we obtain $M,t\models\chi$, as required.
\end{proof}


\noindent Summing up, then, we have seen that our team-based semantics is inter-definable with standard Kripke semantics for intuitionistic logic. As we will see in Section~\ref{sec:logic}, from this connection it follows immediately that the logic determined by our semantics is indeed intuitionistic logic. Thus, what we did so far is simply to provide an alternative semantic foundation for intuitionistic logic. From our perspective, the reason why this new foundation is interesting is that, being based on the notion of support at a team rather than on the notion of truth at a world, it allows us to introduce questions and dependence formulas into the picture. To this we turn in the next section.


\subsection{Adding questions to intuitionistic logic}
\label{sub:adding questions}

\noindent
As in the classical case, questions will enter the picture via the introduction of the inquisitive disjunction connective, $\lori$. Thus, the full language \L\ of the logic we will consider is the one given by Definition \ref{def:full language}. The semantic clause for inquisitive disjunction remains the same as in the classical case.

\begin{definition}[Support conditions for inquisitive disjunction]\label{defsupportsemantics_indisj}~
\begin{itemize}
\item $M,t\models\phi\lori\psi\iff M,t\models\phi$ or $M,t\models\psi$
\end{itemize}
\end{definition}

%

\noindent It is easy to see that formulas formed by means of $\lori$ are not in general truth-conditional. As before, we will  think of such formulas as questions. In particular, if $\alpha$ is a standard formula, we will write $?\aa$ as an abbreviation for $\aa\lori\neg\aa$.
The following proposition says that, as in the classical case, $?\aa$ captures the question of whether $\aa$ is true or false, which is settled in a team $t$ just in case all the assignment in $t$ agree about the truth-value of $\aa$. 

\begin{proposition}[Support conditions for polar questions]\label{prop:polar questions}~\\
Let $M$ be a Kripke model, $t$ a team in $M$, and $\aa\in\L_!$. We have:
$$M,t\models{?\aa}\iff \forall w,w'\in t: T(w,\aa)\deq T(w',\aa)$$
where $T(w,\aa)\deq T(w',\aa)$ means that the two values are defined and equal.
\end{proposition}

\begin{proof} 
Suppose $M,t\models{?\aa}$. Then $M,t\models\aa$ or $M,t\models\neg\aa$, which, by Proposition \ref{prop:truth-conditionality}, implies that $M,w\models\aa$ for all $w\in t$, or $M,w\models\neg\aa$ for all $w\in t$. Clearly, in both cases $T(w,\aa)$ is defined and has a constant value for all $w\in t$.

Conversely, suppose $T(w,\aa)$ is defined and has a constant value for all $w\in t$. If $T(w,\aa)=1$ for all $w\in t$, then $M,w\models\aa$ for all $w\in t$, which by Proposition \ref{prop:truth-conditionality} implies $M,t\models\aa$, and thereby $M,t\models{?\aa}$. If $T(w,\aa)=0$ for all $w\in t$, then by the same argument we obtain  $M,t\models\neg\aa$, and thereby $M,t\models{?\aa}$ holds as well.
\end{proof}

\noindent Notice that there only is a fact of the matter about whether $\aa$ is true or false in case $\aa$ does have a truth-value in the first place. This cannot be taken for granted in our intuitionistic setting, since possible worlds can in general fail to assign a truth-value to some formulas. To know that the question $?\aa$ has a true answer, one has to know that the standard disjunction $\aa\lor\neg\aa$ is true. We may say that $\aa\lor\neg\aa$ captures the presupposition of the question $?\aa$. The formula $\aa\lor\neg\aa$ is not a tautology in our setting, which means that in the intuitionistic setting---unlike in the classical setting---polar questions have non-trivial presuppositions (for a discussion of the logical notion of presupposition of a question, see \cite{Ciardelli:16}, \S1,3).

\begin{figure}[t]
 \centering
 \subfigure[{$p$}]{\label{fig:p}
 \begin{tikzpicture}[>=latex,scale=0.75]
  \node[scale=0.8,shape=circle,draw=black] (A) at (0,0) {$w_1$};
  \node[scale=0.8,shape=circle,draw=black] (B) at (1,1) {$w_3$};
  \node[scale=0.8,shape=circle,draw=black] (C) at (2,2) {$w_5$};
  \node[scale=0.8,shape=circle,draw=black] (D) at (0,2) {$w_4$};
  \node[scale=0.8,shape=circle,draw=black] (E) at (-1,1) {$w_2$};
  \path (A) edge (B);
  \path (B) edge (C);
  \path (B) edge (D);
  \path (A) edge (E);
  \node[scale=0.8] (I) at (-.5,1) {$p$};
  \node[scale=0.8] (J) at (1.5,2) {$q$};
  \draw[opaque,rounded corners] (-1.6,1.5) rectangle (-.3,.5);
 \end{tikzpicture}
 }
 \hspace{.1cm}
 \subfigure[{$\neg p$}]{\label{fig:notp}
 \begin{tikzpicture}[>=latex,scale=0.75]
  \node[scale=0.8,shape=circle,draw=black] (A) at (0,0) {$w_1$};
  \node[scale=0.8,shape=circle,draw=black] (B) at (1,1) {$w_3$};
  \node[scale=0.8,shape=circle,draw=black] (C) at (2,2) {$w_5$};
  \node[scale=0.8,shape=circle,draw=black] (D) at (0,2) {$w_4$};
  \node[scale=0.8,shape=circle,draw=black] (E) at (-1,1) {$w_2$};
  \path (A) edge (B);
  \path (B) edge (C);
  \path (B) edge (D);
  \path (A) edge (E);
  \node[scale=0.8] (I) at (-.5,1) {$p$};
  \node[scale=0.8] (J) at (1.5,2) {$q$};
  \draw[opaque,rounded corners=3mm] (1,.15) -- (-.95,2.45) -- (2.95,2.45) -- cycle; \end{tikzpicture}
 }
   \hspace{.1cm}
 \subfigure[{$p\vee\neg p$}]{\label{fig:lemp}
 \begin{tikzpicture}[>=latex,scale=0.75]
  \node[scale=0.8,shape=circle,draw=black] (A) at (0,0) {$w_1$};
  \node[scale=0.8,shape=circle,draw=black] (B) at (1,1) {$w_3$};
  \node[scale=0.8,shape=circle,draw=black] (C) at (2,2) {$w_5$};
  \node[scale=0.8,shape=circle,draw=black] (D) at (0,2) {$w_4$};
  \node[scale=0.8,shape=circle,draw=black] (E) at (-1,1) {$w_2$};
  \path (A) edge (B);
  \path (B) edge (C);
  \path (B) edge (D);
  \path (A) edge (E);
  \node[scale=0.8] (I) at (-.5,1) {$p$};
  \node[scale=0.8] (J) at (1.5,2) {$q$};
    \draw[opaque,rounded corners] (0,.5) -- (-1.65,.5) -- (-1.65,1.2) -- (-.35,2.5) -- (2.5,2.5) -- (2.5,1.5) --(1.5,.5) -- (0,.5); 
 \end{tikzpicture}
 }
 \hspace{.1cm}
 \subfigure[{$q$}]{\label{fig:q}
 \begin{tikzpicture}[>=latex,scale=0.75]
  \node[scale=0.8,shape=circle,draw=black] (A) at (0,0) {$w_1$};
  \node[scale=0.8,shape=circle,draw=black] (B) at (1,1) {$w_3$};
  \node[scale=0.8,shape=circle,draw=black] (C) at (2,2) {$w_5$};
  \node[scale=0.8,shape=circle,draw=black] (D) at (0,2) {$w_4$};
  \node[scale=0.8,shape=circle,draw=black] (E) at (-1,1) {$w_2$};
  \path (A) edge (B);
  \path (B) edge (C);
  \path (B) edge (D);
  \path (A) edge (E);
  \node[scale=0.8] (I) at (-.5,1) {$p$};
  \node[scale=0.8] (J) at (1.5,2) {$q$};
  \draw[opaque,rounded corners] (1.25,2.5) rectangle (2.55,1.5);
 \end{tikzpicture}
 }
 \hspace{.1cm}
 \subfigure[{$\neg q$}]{\label{fig:notq}
 \begin{tikzpicture}[>=latex,scale=0.75]
  \node[scale=0.8,shape=circle,draw=black] (A) at (0,0) {$w_1$};
  \node[scale=0.8,shape=circle,draw=black] (B) at (1,1) {$w_3$};
  \node[scale=0.8,shape=circle,draw=black] (C) at (2,2) {$w_5$};
  \node[scale=0.8,shape=circle,draw=black] (D) at (0,2) {$w_4$};
  \node[scale=0.8,shape=circle,draw=black] (E) at (-1,1) {$w_2$};
  \path (A) edge (B);
  \path (B) edge (C);
  \path (B) edge (D);
  \path (A) edge (E);
  \node[scale=0.8] (I) at (-.5,1) {$p$};
  \node[scale=0.8] (J) at (1.5,2) {$q$};
  \draw[opaque,rounded corners,rotate=45] (-.5,.7) rectangle (1.9,2);
 \end{tikzpicture}
 }
  \hspace{.1cm}
 \subfigure[{$q\vee\neg q$}]{\label{fig:lemq}
 \begin{tikzpicture}[>=latex,scale=0.75]
  \node[scale=0.8,shape=circle,draw=black] (A) at (0,0) {$w_1$};
  \node[scale=0.8,shape=circle,draw=black] (B) at (1,1) {$w_3$};
  \node[scale=0.8,shape=circle,draw=black] (C) at (2,2) {$w_5$};
  \node[scale=0.8,shape=circle,draw=black] (D) at (0,2) {$w_4$};
  \node[scale=0.8,shape=circle,draw=black] (E) at (-1,1) {$w_2$};
  \path (A) edge (B);
  \path (B) edge (C);
  \path (B) edge (D);
  \path (A) edge (E);
  \node[scale=0.8] (I) at (-.5,1) {$p$};
  \node[scale=0.8] (J) at (1.5,2) {$q$};
  \draw[opaque,rounded corners] (-1.8,.8) -- (-.35,2.5) -- (2.5,2.5) -- (2.5,1.5) --(0.2,1.5) -- (-.9,.07) -- cycle; 
 \end{tikzpicture}
 }
 \hspace{.1cm}
 \subfigure[{$?p$}]{\label{fig:questionp}
 \begin{tikzpicture}[>=latex,scale=0.75]
  \node[scale=0.8,shape=circle,draw=black] (A) at (0,0) {$w_1$};
  \node[scale=0.8,shape=circle,draw=black] (B) at (1,1) {$w_3$};
  \node[scale=0.8,shape=circle,draw=black] (C) at (2,2) {$w_5$};
  \node[scale=0.8,shape=circle,draw=black] (D) at (0,2) {$w_4$};
  \node[scale=0.8,shape=circle,draw=black] (E) at (-1,1) {$w_2$};
  \path (A) edge (B);
  \path (B) edge (C);
  \path (B) edge (D);
  \path (A) edge (E);
  \node[scale=0.8] (I) at (-.5,1) {$p$};
  \node[scale=0.8] (J) at (1.5,2) {$q$};
  \draw[opaque,rounded corners] (-1.6,1.5) rectangle (-.3,.5);
  \draw[opaque,rounded corners=3mm] (1,.15) -- (-.95,2.45) -- (2.95,2.45) -- cycle;
 \end{tikzpicture}
 }
 \hspace{.1cm}
 \subfigure[{$?q$}]{\label{fig:questionq}
 \begin{tikzpicture}[>=latex,scale=0.75]
  \node[scale=0.8,shape=circle,draw=black] (A) at (0,0) {$w_1$};
  \node[scale=0.8,shape=circle,draw=black] (B) at (1,1) {$w_3$};
  \node[scale=0.8,shape=circle,draw=black] (C) at (2,2) {$w_5$};
  \node[scale=0.8,shape=circle,draw=black] (D) at (0,2) {$w_4$};
  \node[scale=0.8,shape=circle,draw=black] (E) at (-1,1) {$w_2$};
  \path (A) edge (B);
  \path (B) edge (C);
  \path (B) edge (D);
  \path (A) edge (E);
  \node[scale=0.8] (I) at (-.5,1) {$p$};
  \node[scale=0.8] (J) at (1.5,2) {$q$};
  \draw[opaque,rounded corners] (1.25,2.5) rectangle (2.55,1.5);
  \draw[opaque,rounded corners,rotate=45] (-.5,.7) rectangle (1.9,2);
 \end{tikzpicture}
 }
 \hspace{.1cm}
\subfigure[{${?p}\wedge{?q}$}]{\label{fig:quepandqueq}
 \begin{tikzpicture}[>=latex,scale=0.75]
  \node[scale=0.8,shape=circle,draw=black] (A) at (0,0) {$w_1$};
  \node[scale=0.8,shape=circle,draw=black] (B) at (1,1) {$w_3$};
  \node[scale=0.8,shape=circle,draw=black] (C) at (2,2) {$w_5$};
  \node[scale=0.8,shape=circle,draw=black] (D) at (0,2) {$w_4$};
  \node[scale=0.8,shape=circle,draw=black] (E) at (-1,1) {$w_2$};
  \path (A) edge (B);
  \path (B) edge (C);
  \path (B) edge (D);
  \path (A) edge (E);
  \node[scale=0.8] (I) at (-.5,1) {$p$};
  \node[scale=0.8] (J) at (1.5,2) {$q$};
  \draw[opaque,rounded corners] (-1.6,1.5) rectangle (-.3,.5);
  \draw[opaque,rounded corners] (-.6,2.5) rectangle (.6,1.54);
  \draw[opaque,rounded corners] (1.25,2.5) rectangle (2.55,1.5);
 \end{tikzpicture}
 }
 \caption{The maximal support teams for some propositional formulas.}
 \label{fig:intteampolar}
\end{figure}
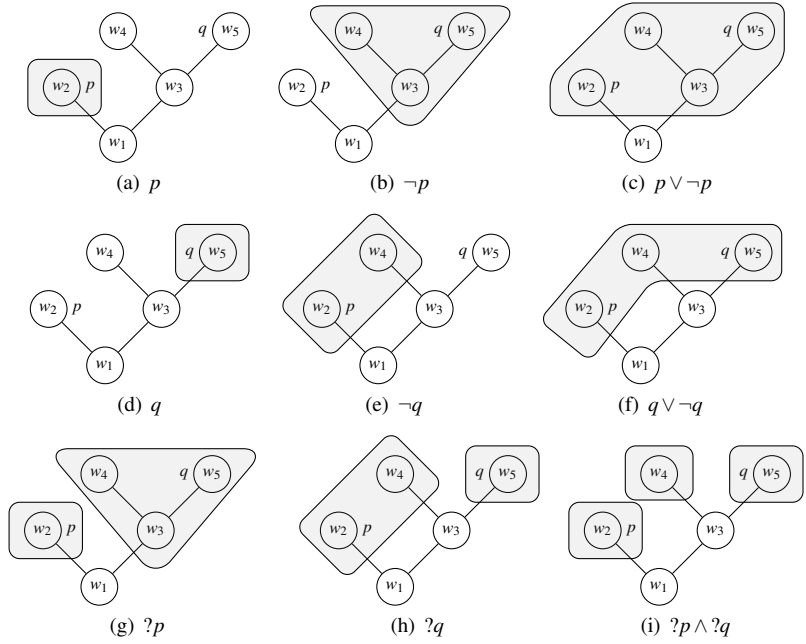

\begin{example} To get a concrete impression of the interpretation of statements and questions in \inqi, consider Figure~\ref{fig:intteampolar}. Since statements are truth-conditional by definition, we always have a unique largest team where a given statement is supported, which coincides with the set of all worlds where the statement is true. For standard formulas, Proposition \ref{prop:intuitionistic truth-conditions} ensures that this team coincides with the truth-set of the statement as given by standard Kripke semantics. This is illustrated by figures \ref{fig:p}-\ref{fig:lemq}.  By contrast, for questions we typically have multiple maximal supporting teams, corresponding to the different ways for the question to be resolved in the given model. This is illustrated by figures \ref{fig:questionp}-\ref{fig:quepandqueq}. Notice the difference between the meaning of the standard disjunction $p\lor\neg p$ and the meaning of the corresponding inquisitive disjunction $p\lori\neg p$, abbreviated as ${?p}$. To support $p\lor\neg p$, the information available in the team just has to ensure that $p$ has a definite truth-value. This is also necessary in order to support $?p$, but it is not sufficient: in this case, the information available in the team should also determine whether $p$ is true or false. The situation is analogous for the standard disjunction $q\lor\neg q$ and the inquisitive disjunction $?q$. Finally, notice that, as in the classical case, standard connectives are allowed to operate on questions. This is illustrated by Figure \ref{fig:quepandqueq}, which depicts the meaning of the conjunctive question ${?p}\land{?q}$: as we expect, this question is only supported by those teams that determine both the truth-value of $p$, and the truth-value of $q$.
\end{example}

\subsection{General features of the semantics}  
\label{sub:semantic features}
Now let us turn to some general feature of the semantics. As in the classical case, the semantics is persistent, meaning that if a formula $\phi$ is supported in a team $t$, it remains supported in any extension of $t$; however, now the extensions of $t$ are not only the subsets of $t$, but all teams $s\subseteq R[t]$. Moreover, as in the classical case, all formulas are supported by the empty team, which models the state of inconsistent information.

\begin{proposition}[Persistency and empty team property]\label{prop:persistency}~
\begin{itemize}
\item Persistency property: if $M,t\models\phi$ and $s\subseteq R[t]$, then $M,s\models\phi$
\item Empty team property: $M,\emptyset\models\phi$ for all $\phi\in\L$
\end{itemize}
\end{proposition}

\begin{proof} Straightforward, by induction on $\phi$.
\end{proof}

\noindent An immediate consequence of persistency is the following property of our system: if a formula is supported by a team $t$, it is also supported by the whole up-set $R[t]$ generated by $t$, and conversely.

\begin{proposition}[Up-set property]\label{prop:up-set}~\\
$M,t\models\phi\iff M,R[t]\models\phi$
\end{proposition}

\begin{proof} Since $R[t]\subseteq R[t]$ and $t\subseteq R[t]=R[R[t]]$, the result follows from persistency.
\end{proof}


\medskip
\noindent This proposition implies that, from the perspective of our semantics, the only thing that matters about a team $t$ is the up-set $R[t]$ that is generates. 
This can be used to show a preservation result in the opposite direction, from a team $t$ to a sub-team: if the relation $R$ is well-founded (i.e., if every $t\subseteq W$ contains an $R$-minimal element) then a team $t$ can be replaced by the set $\min(t)$ of its $R$-minimal elements without affecting support. 

\begin{proposition}[Minimal-set property]\label{prop:minimal-set}~\\
For any Kripke model $M=\<W,R,V\>$ where $R$ is well-founded, for any 
team $t$ and formula $\phi$: 
$$
M,t \models\phi \iff M,\min(t)\models\phi.
$$  
where $\min(t)=\{ w \in t \mid \text{for all }v\in t: vRw\text{ implies }v=w\}$.
\end{proposition}
\begin{proof} This follows from the previous proposition since, if $R$ is well-founded, then $R[t]=R[\min(t)]$.
\end{proof}

\noindent Another important fact about the semantics is that, for the purposes of evaluation at~$t$, the part of the model which lies outside the up-set $R[t]$ is irrelevant. If $M$ is a model and $t$ is a team in $M$, let us denote by $M_t$ the sub-model generated by $t$, i.e.,  
the restriction of $M$ to $R[t]$. 
Then, we have the following general property.

\begin{proposition}[Restriction property]\label{prop:locality}~\\
$M,t\models\phi\iff M_t,t\models\phi$
\end{proposition}

\begin{proof} Straightforward, by induction on $\phi$. 
\end{proof}

\noindent
Proposition \ref{prop:up-set} implies that every formula that can be falsified in our semantics can be falsified on an up-set. 
In a similar spirit, the following result states that every formula that can be falsified can be falsified at a single world.

\begin{proposition}[Single world property]\label{prop:single world}~\\
If there are a Kripke model $M$ and a team $t$ such that $M,t\not\models\phi$, then there is a model $M'$ and a world $w$ such that $M',w\not\models\phi$.
\end{proposition}

\begin{proof} If $M=\<W,R,V\>$, let $M'=\<W',R',V'\>$ be the model obtained by adding a new root to $M$ and making all propositional letters false at this root. More precisely, $W'=W\cup\{w\}$, where $w$ is a new world not in $W$. $V'$ coincides with $V$ on all worlds in $W$, and it makes every atom false at $w$. Finally, $R'=R\cup\{\<w,w'\>\mid w'\in W'\}$.

Now, the restriction of the models $M$ and $M'$ to $t$ is the same: $M'_t=M_t$. By the restriction property (Proposition \ref{prop:locality}), this gives the following:
$$M',t\models\phi\iff M'_t,t\models\phi\iff M_t,t\models\phi\iff M,t\models\phi$$
Since $M,t\not\models\phi$ by assumption, also $M',t\not\models\phi$. But by definition of $R'$, $t\subseteq R'[w]$. Thus, persistency implies $M',w\not\models\phi$.
\end{proof}

\noindent
This means that, in \ilqd, all formulas that are true at all worlds are also supported by all teams, and thus logically valid. This is in stark contrast to what happens in the classical case: in \lqd, $?p$ is true at all worlds in all models, but it is not logically valid. Indeed, in \lqd\ we have that from the perspective of a single world $w$, standard disjunction $\lor$ and inquisitive disjunction $\lori$ are indistinguishable (Proposition \ref{prop:standard variant}). So, if the single world property held in \lqd, there would not be any logical difference between standard disjunction and inquisitive disjunction. By contrast, in \ilqd\ the single world property holds, but as we will see in Section \ref{sub:truth-conditional}, even single worlds are in general sensitive to the difference between standard and inquisitive disjunction.

%
%


\subsection{Expressing dependencies in \ilqd}
\label{sub:dependency}

\noindent As in the classical case, it is natural to say that a question $\mu$ determines another question $\nu$ in a team $t$ of a model $M$ in case extending $t$ so as to settle $\mu$ is bound to lead to a state where $\nu$ is settled as well.
$$\mu\text{ determines }\nu\text{ in }t\iff \forall s\subseteq R[t]:M,s\models\mu\text{ implies }M,s\models\nu$$
The condition on the right is precisely what is required for the team $t$ to support of the implication $\mu\to\nu$. Thus, we have:
$$\mu\text{ determines }\nu\text{ in }t\iff M,t\models\mu\to\nu$$
Just as in the classical case, then, in \ilqd\ implications between questions express the existence of certain dependency relations. As an example, consider the formula $\dep(p,q)$, which we took to be an abbreviation for ${?p}\to {?q}$. As in the classical case, this expresses the fact that relative to the given team $t$, the truth-value of $q$ is functionally determined by the truth-value of $p$. However, 
now this dependency has to be robust enough to take into account the fact that the worlds in $t$ may also become more defined. Concretely, we have the following characterization (where, recall, $T(w,p)\!\deq\! T(w',p)$ means that the two values are defined and equal).

\begin{proposition}\label{prop:dependency char1}~\\
$M,t\models{?p}\to {?q}\iff \forall w,w'\!\in\! R[t], T(w,p)\!\deq\! T(w',p)\text{ implies }T(w,q)\!\deq\! T(w',q)$
\end{proposition}

\begin{proof} Suppose $M,t\models{?p}\to {?q}$. Take worlds $w,w'\in R[t]$ s.t.\ $T(w,p)\!\deq\! T(w',p)$. This means that either $p$ is true at both worlds, or $\neg p$ is. In the former case, $M,\{w,w'\}\models p$; in the latter case, $M,\{w,w'\}\models\neg p$; in either case, $M,\{w,w'\}\models{?p}$. Since $\{w,w'\}\subseteq R[t]$ and $M,t\models{?p}\to{?q}$, this implies $M,\{w,w'\}\models{?q}$. This means that either $M,\{w,w'\}\models p$, or $M,\{w,w'\}\models\neg p$. In the former case, by persistency $p$ is true at both $w$ and $w'$. In the latter case, by persistency $p$ is false at both $w$ and $w'$. In either case, we thus have $T(w,q)\deq T(w',q)$. 

Conversely, suppose $\forall w,w'\!\in\! R[t], T(w,p)\!\deq\! T(w',p)\text{ implies }T(w,q)\!\deq\! T(w',q)$. Consider a team $s\subseteq R[t]$ such that $M,s\models{?p}$. By Proposition \ref{prop:polar questions}, this means that the truth-value of $p$ is defined and constant throughout $s$. Now take any worlds $w,w'\in s$. We must have $T(w,p)\deq T(w',p)$. Since $w,w'\in s\subseteq R[t]$, our assumption applies, giving $T(w,q)\deq T(w',q)$. Since $w$ and $w'$ were arbitrary in $s$, this shows that the truth-value of $q$ is defined and constant in $s$, which again by Proposition \ref{prop:polar questions} implies $M,s\models{?q}$. This shows that $M,t\models{?p}\to{?q}$.
\end{proof}

\noindent An equivalent characterization of the support conditions of ${?p}\to {?q}$ is the following, which brings out even more explicitly the fact that this formula captures the existence of a functional dependency between the truth-values of $p$ and $q$. The proof is straightforward, given the previous proposition.

\begin{proposition}\label{prop:dependency char2}~\\
$M,t\models{?p}\to {?q}\;\iff\;\text{there is a function }f:\{0,1\}\to\{0,1\}\text{ s.t. }\forall w\in R[t]$\\
$\phantom{M,t\models{?p}\to {?q}\;\iff\;}\;\text{if }T(w,p)\text{ is defined, then }T(w,q)=f(T(w,p))$
\end{proposition}


\begin{figure}
 \centering
 \begin{tikzpicture}[scale=0.7]
  \node[shape=circle,draw=black] (A) at (0,0) {$w$};
  \node[shape=circle,draw=black] (B) at (2,2) {$v$};
  \node[shape=circle,draw=black] (E) at (-2,2) {$u$};
  \path (A) edge (B);
  \path (A) edge (E);
  \node (H) at (2.7,2)  {$p$};
  \node (J) at (-3,2)  {$p,q$};
 \end{tikzpicture}
 \caption{A Kripke model where the dependency ${?p}\to{?q}$ fails at the singleton team~$\{w\}$.}
 \label{fig:idep_singleton}
\end{figure}
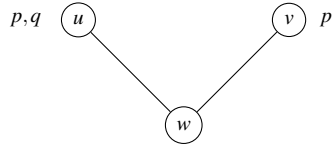

\noindent Notice an interesting difference with the classical setting: in the present setting, a dependency may fail to hold even with respect to a singleton team. For instance consider the model in Figure \ref{fig:idep_singleton} and the singleton team $\{w\}$:
the dependency ${?p}\to{?q}$ does not hold relative to this singleton, because there is no function $f:\{0,1\}\to\{0,1\}$ which, as soon as the truth value of $p$ becomes defined, allows us to derive from this the truth-value of $q$. Indeed, $p$ becomes true at both the proper successors of $w$, but $q$ is true in the one and false in the other.  


Of course, this discussion extends easily to dependencies that involve more than two questions. 
We say that $\nu$ is determined by $\mu_1,\dots,\mu_n$ relative to a team $t$ in case extending $t$ so as to settle the questions $\mu_1,\dots,\mu_n$ is bound to lead to a team that settles $\nu$. This is precisely what is expressed by the implication $\mu_1\land\dots\land\mu_n\to\nu$.

In particular, if $\aa_1,\dots,\aa_n,\bb$ are standard formulas, the fact that the truth-value of $\bb$ is functionally determined by the values of $\aa_1,\dots,\aa_n$ is captured by the implication ${?\aa_1}\land\dots\land{?\aa_n}\to{?\bb}$. Propositions \ref{prop:dependency char1} and  \ref{prop:dependency char2} generalize straightforwardly to this case, as follows.

\begin{proposition}~\\
Let $M$ be a Kripke model, $t$ a team, and $\aa_1,\dots,\aa_n,\bb$ standard formulas. Then the following are equivalent:
\begin{enumerate}
\item $M,t\models{?\aa_1}\land\dots\land{?\aa_n}\to {?\bb}$
\item $\forall w,w'\!\in\! R[t]:\text{if } T(w,\aa_i)\deq T(w',\aa_i)\text{ for all }i\le n\text{ then }T(w,\bb)\deq T(w',\bb)$
\item $\text{There is }f:\{0,1\}^n\overset{\phantom{*}}{\to}\{0,1\}\text{ such that }T(w,\bb)=f(T(w,\aa_1),\dots,T(w,\aa_n))$\\
$\text{whenever }w\in R[t]\text{ and }T(w,\aa_1),\dots,T(w,\aa_n)\text{ are defined}$
\end{enumerate}
\end{proposition}



\noindent
Working in the classical setting, Abramsky and V\"a\"an\"anen  \cite{AbramskyVaananen:09} noticed that, once dependencies are expressed by means of implication, the Armstrong axioms for functional dependency  \citep{Armstrong:74}, well-known in the database theory literature, amount simply to the axioms of intuitionistic implicational logic. Since dependencies in the intuitionistic setting are still captured by implications, this means that Armstrong's axioms are still valid for dependencies in the intuitionistic setting.

\subsection{Two alternative setups for the semantics}
\label{sub:alternative setups}

We close this section by considering two alternative ways to set up a semantics for \inqi\ 
which turn out to be formally equivalent to the proposal we have spelled out in this section.

\subsubsection{Teams as sets of rooted Kripke models} In inquisitive logic, a team is considered to be a set of possible states of affairs, which means a set of models for the underlying logic of statements. What exactly these models are depends on the specific logic that we start out with. In the classical propositional setting, these models are essentially propositional valuations. In the intuitionistic setting, the natural counterpart of a valuation is a rooted Kripke model. 
Therefore, one might take a team $T$ to be a set $\{M_i\mid i\in\mathcal{I}\}$ of rooted Kripke models, rather than a set of nodes in a single Kripke model. The clauses defining the support relation $T\models\phi$ (without a background model, in this case) will be the same as those in in Definition~\ref{defsupportsemantics}, modulo adapting the definition of the extension relation between teams in the natural way.\footnote{Namely, we can regard a team $T'$ as an extension of $T$ if every model in $T'$ is a generated sub-model of some model in $T$. This means that, as in our approach, a team $T$ can be extended not only by discarding some candidate $M\in T$, but also by refining some $M\in T$ in one or more ways.}

This alternative approach results in the same logic as our approach, for the following reason. On the one hand, given a Kripke model $M$ and a team $t$ in our sense, we can take the set $T_t$ consisting of the sub-models generated by the worlds in $t$: $T_t=\{M_w\mid w\in t\}$: this is a set of rooted Kripke models, and thus a team in the alternative sense. Vice versa, given a set $T=\{M_i\mid i\in\mathcal{I}\}$ of Kripke models with roots $\{w_i\mid i\in\mathcal{I}\}$, we can consider the set $t_T=\{w_i\mid i\in\mathcal{I}\}$ and regard it as a team within the Kripke model $M_T=\biguplus_{i\in\mathcal{I}}M_i$.  One can then verify that these two transformations preserve the semantics, in the sense that we have $M,t\models\phi\iff T_t\models\phi$ and $T\models\phi\iff M_T,t_T\models\phi$.

\subsubsection{Restricting to upwards-closed teams} Another natural perspective is to base the semantics on information states construed not as arbitrary teams, but rather as \emph{upwards-closed} teams, i.e., teams $t$ such that whenever $w\in t$ and $wRv$, we have $v\in t$. This view can be justified based on the idea that an information state consists of those worlds that cannot be ruled out on the basis of the available information: it is natural to assume that if $w$ cannot be ruled out and if $w$ may be refined to obtain $v$, then $v$ cannot be ruled out either. This means that, if an information state is modeled as the set $t$ of worlds which it does not rule out, then $t$ must be upward-closed.

%


Notice that if $u$ is an upward-closed team, then $R[u]=u$. So, relative to upward-closed teams, the extension relation $u'\subseteq R[u]$ boils down to $u'\subseteq u$, just as in the classical case. Thus, in this perspective the move from classical to intuitionistic logic is obtained not by modifying the extension relation between teams, but rather by restricting the set of teams which are regarded as information states, and which are thus available in the semantics. This goes in the direction of the approach taken by Pun\v coch\'a\v r in \cite{Puncochar:15generalization}, which will be discussed in Section \ref{sec:related work}.

A modified support relation $M,u\models'\phi$ can be defined simply by restricting the clauses in Definition \ref{defsupportsemantics} in the obvious way to upward-closed teams. Again, it is easy to see that this modification does not lead to a semantics which is genuinely different from the one we have introduced in this section. In one direction, we can show that for upwards-closed teams $u$, the restricted semantics gives the same results as our general semantics: $M,u\models'\phi\iff M,u\models\phi$. Vice versa, our general semantics can be simulated in the restricted setting by looking at the upward-closure $R[t]$ of the team $t$ under consideration: $M,t\models\phi\iff M,R[t]\models'\phi$.



\section{Intuitionistic inquisitive logic}
\label{sec:logic}

In the previous section, we have seen how standard Kripke semantics for intuitionistic logic can be generalized so that we can interpret not only statements, but also questions and formulas that express dependencies. Let us now turn to the investigation of the logic that arises from this system. We start in Section \ref{sub:entailment} by establishing some fundamental features of this logic. In Section \ref{sub:normal form} we show that the inquisitive normal form result that we have seen in the classical setting carries over to the intuitionistic case. In Section \ref{sub:truth-conditional} we look at some characterizations of truth-conditional formulas in \ilqd. Finally, in Section \ref{sub:axiomatization} we establish a sound and complete axiomatization of our logic.

\subsection{Entailment in \ilqd}
\label{sub:entailment}

Entailment in the logic \ilqd\ is defined in the natural way.

\begin{definition}[Entailment]~\\
Let $\Phi\cup\{\psi\}\subseteq\L$. We say that $\Phi$ entails $\psi$ in \ilqd, notation $\Phi\entilqd\psi$, if for any Kripke model $M$ and team $t$ in $M$: $M,t\models\phi$ for all $\phi\in\Phi$ implies $M,t\models\psi$.
\end{definition}

\noindent When no confusion arises, we will allow ourselves to drop the subscript \ilqd. 
The first thing that we can verify is that, as we expect from the discussion in Section \ref{sub:intuitionistic support}, the system \ilqd\ is a conservative extension of intuitionistic propositional logic (\ipl). This means that, on standard formulas, entailment in \ilqd\ is just entailment in \ipl.

\begin{theorem}[Conservativity over intuitionistic logic]\label{conservativity_ipl}~\\
For $\Phi\cup\{\psi\}\subseteq\L_!$,  $\Phi\entilqd\psi\iff \Phi\models_{\textsf{IPL}}\psi$.
\end{theorem}

\begin{proof} Suppose $\Phi\entilqd\psi$. This means that whenever all formulas in $\Phi$ are supported, so is $\psi$. Since by definition truth is a particular case of support, this implies that whenever all formulas in $\Phi$ are true, so is $\psi$. On the other hand, Proposition \ref{prop:intuitionistic truth-conditions} ensures that for standard formulas, our notion of truth coincides with notion given by Kripke semantics. Since Kripke semantics characterizes \ipl, we have $\Phi\models_{\ipl}\psi$.

For the converse, suppose $\Phi\not\models_{\textsf{IPL}}\psi$. This means that there exists a Kripke model $M$ and a team $t$ in $M$ such that $t$ supports all formulas in $\Phi$ but does not support $\phi$. Since $\psi$ is a standard formula, by Proposition \ref{prop:truth-conditionality}, there must then be a world $w\in t$ which does not make $\psi$ true. By persistency, $w$ makes true all formulas in $\Phi$. Since truth coincides with the notion given by standard Kripke semantics, this shows that $\Phi\not\models_\ipl\psi$.
\end{proof}

\noindent Thus, as we wanted, our logic \ilqd\ integrates questions and dependence formulas within an intuitionistic logic of statements. As in the classical case, the presence of questions allows us to capture some interesting logical notions as cases of entailment which involve questions. In particular, if $\Gamma$ is a set of statements and $\mu$ is a question, then the entailment $\Gamma\models\mu$ holds if and only if the information provided by the statements in $\Gamma$ suffices to settle the question $\mu$. This is as in the classical case. However, the fact that \ilqd\ is based on intuitionistic logic makes it harder for a set of statements to qualify as settling a question. For instance, the entailment  $\neg\neg p\models {?p}$ is valid in the classical system \lqd\ described above, but it is invalid in our system \ilqd: intuitionistically, the statement $\neg\neg p$ does not provide sufficient information to resolve the question $?p$, since the truth of $\neg\neg p$ does not imply the truth of $p$.



If $\Gamma$ is a set of statements, $\Lambda$ a non-empty set of questions, and $\mu$ a question, then the entailment $\Gamma,\Lambda\models\mu$ holds if and only if given the statements in $\Gamma$, the question $\nu$ is completely determined by the questions in $\Lambda$. For instance, just as in the classical case, the entailment $p\leftrightarrow q,{?p}\models{?q}$ is valid, witnessing that given the assumption $p\leftrightarrow q$, the truth-value of $p$ determines the truth-value of $q$. Like in the case of answerhood, however, whether a dependency holds is a matter that hinges partly on the underlying logic of statements. For a basic example, the entailment $?\neg p\models{?p}$ is valid in the classical case, but not in \ilqd. Intuitionistically, the question $?\neg p$ does not logically determine the question $?p$, since one may settle $?\neg p$ by establishing 
$\neg\neg p$, which as we saw would not settle the question $?p$. 

Summing up, then, entailment in \ilqd\ provides an intuitionistic perspective on a number of interesting logical notions, which include not only the standard logical consequence relation between statements, but also the relation of logical answerhood between statements and questions, and the relation of logical dependency between questions. To the best of our knowledge, these logical relations have not been investigated before in the context of intuitionistic logic.

Let us now turn to the formal features of the entailment relation in \ilqd. An unsurprising but important feature of our logic which is worth stating explicitly, since it is often used without mentioning it later on, is the connection between entailment and implication---the semantic analogue of the deduction theorem. 

\begin{proposition}\label{prop:deduction} For all $\Phi\cup\{\psi,\chi\}\subseteq\L$, $\;\Phi,\psi\models\chi\iff \Phi\models\psi\to\chi$.
\end{proposition}

\begin{proof} Straightforward, using the persistency property of the semantics.
\end{proof}

\noindent
Importantly, the disjunction property, which is characteristic of intuitionistic logic, is preserved in this broader logical setting.

\begin{theorem}[Disjunction property for $\lor$]\label{theor:disjunction}~\\
For any formulas $\phi,\psi\in\L$, $\models\phi\lor\psi$ implies $\models\phi$ or $\models\psi$.
\end{theorem}

\begin{proof}
Suppose $\not\models\phi$ and $\not\models\psi$. Then, $M_1,t_1\not\models\phi$ and $M_2,t_2\not\models\psi$ for some Kripke models $M_1$ and $M_2$, and teams $t_1$ and $t_2$. Let $M$ be a Kripke model constructed by putting one point $w$ below the disjoint union $M_1\uplus M_2$ of the two models $M_1$ and $M_2$ and making no propositional variable true at $w$. We know that $M_1,t_1\models\phi$, which by Proposition \ref{prop:locality} implies $M,t_1\models\phi$. For the same reason, $M,t_2\models\psi$. Let $R$ be the partial order of $M$. Since $t_1,t_2\subseteq R[w]$, by persistency, we have $M,w\not\models\phi$ and $M,w\not\models\psi$, which imply $M,w\not\models\phi\lor\psi$. Hence, $\not\models\phi\lor\psi$.
\end{proof}

\noindent
Interestingly, the same property is shared by inquisitive disjunction. In fact, inquisitive disjunction satisfies an even more general property, the \emph{split property} that we have already encountered in the context of \lqd\ (Proposition \ref{prop:split classical}).

\begin{theorem}[Split property for $\lori$]\label{theor:split}~\\
If $\Gamma$ is a set of truth-conditional formulas, $\Gamma\models\phi\lori\psi$ implies $\Gamma\models\phi$ or $\Gamma\models\psi$.
\end{theorem}

\begin{proof}
By contraposition, suppose  $\Gamma\not\models\phi$ and  $\Gamma\not\models\psi$. Then we have two Kripke models $M_1,M_2$ and two teams $t_1,t_2$ such that:
\begin{itemize}
\item  $M_1,t_1\models\gamma$ for all $\gamma\in\Gamma$,  
but  $M_1,t_1\not\models\phi$
\item  $M_2,t_2\models\gamma$ for all $\gamma\in\Gamma$,  
but $M_2,t_2\not\models\psi$
\end{itemize}
Now consider the model $M:=M_1\uplus M_2$ defined in the natural way as the disjoint union of $M_1$ and $M_2$, and consider the team $t:=t_1\uplus t_2$. Take any $\gamma\in\Gamma$. We know that $M_1,t_1\models\gamma$, which by the restriction property (Proposition \ref{prop:locality}) implies $M,t_1\models\gamma$. For the same reason, we know that $M,t_2\models\gamma$. By persistency, this implies that $\gamma$ is true at any world $w\in t_1\uplus t_2=t$. Since $\gamma$ is truth-conditional, it follows that $M,t\models\gamma$.

On the other hand, we know $M_1,t_1\not\models\phi$. By the restriction property, it follows that $M,t_1\not\models\phi$, which by persistency implies $M,t\not\models\phi$. Proceeding analogously we can conclude that $M,t\not\models\psi$. Hence, $M,t\not\models\phi\lori\psi$.

We have thus found a team which supports all formulas in $\Gamma$ but does not support $\phi\lori\psi$, which shows that $\Gamma\not\models\phi\lori\psi$.
\end{proof}

\noindent This property ensures that, in the intuitionistic case as well, a set of statements can only logically resolve a question by entailing a particular answer to the question. By taking $\Gamma=\emptyset$, we obtain the disjunction property for $\lori$ as a special case.

\begin{corollary}[Disjunction property for $\lori$]\label{cor:dp for lori}~\\ 
For any formulas $\phi,\psi\in\L$, $\models\phi\lori\psi$ implies $\models\phi$ or $\models\psi$.
\end{corollary}

\noindent
Combining Theorem \ref{theor:disjunction} and Corollary \ref{cor:dp for lori}, we obtain the following corollary. 

\begin{corollary}~\\ 
For any formulas $\phi_1,\dots,\phi_n\in\L$, $\;\models\phi_1\lor\dots\lor\phi_n\;\iff\;\, \models\phi_1\lori\dots\lori\phi_n$.
\end{corollary}

\noindent In spite of what this corollary may suggest, however, the two disjunction operators have different logical properties. To see this, let us first prove that, as in the classical case, the split property also holds from the perspective of the object language.

\begin{theorem}[Internal split property for $\lori$]\label{theor:internal split}~\\
If $\alpha$ is a truth-conditional formula, $\alpha\to\phi\lori\psi\models(\aa\to\phi)\lori(\aa\to\psi)$.
\end{theorem}

\begin{proof} Suppose $M,t\not\models\aa\to\phi$ and $M,t\not\models\aa\to\psi$. We will prove $M,t\not\models\aa\to\phi\lori\psi$. Since $M,t\not\models\aa\to\phi$, we have a team $s_1\subseteq R[t]$ such that $M,s_1\models\aa$ but $M,s_1\not\models\phi$. Similarly, since $M,t\not\models\aa\to\psi$, we have a team $s_2\subseteq R[t]$ such that $M,s_2\models\aa$ but $M,s_2\not\models\phi$. Now consider the team $s_1\cup s_2$. Obviously, $s_1\cup s_2\subseteq R[t]$. By persistency, $\aa$ must be true at any world in $s_1\cup s_2$, and since $\aa$ is truth-conditional this implies $M,s_1\cup s_2\models\aa$. However, by persistency we have $M,s_1\cup s_2\not\models\phi$ and $M,s_1\cup s_2\not\models\psi$, which implies $M,s_1\cup s_2\not\models\phi\lori\psi$. Finally, since there exists a team $s\subseteq R[t]$ such that $M,s\models\aa$ but $M,s\not\models\phi\lori\psi$, we have $M,t\not\models\aa\to\phi\lori\psi$.
\end{proof}

\noindent As we will see, the local split property plays an important role in establishing a normal form result and a completeness result for \ilqd. It also allows us to show that standard disjunction and inquisitive disjunction have different logical properties. It follows immediately by the internal split property and Proposition \ref{prop:deduction} that the formula $(p\to q\lori r)\to(p\to q)\lori(p\to r)$ is logically valid. By contrast, the formula $(p\to q\lor r)\to(p\to q)\lor(p\to r)$ is not valid, as witnessed by the model in Figure~\ref{fig:lor vs lori}. 
Thus, some of the principles that hold for inquisitive disjunction do not hold for standard disjunction.

\begin{figure}
 \centering
 \begin{tikzpicture}[scale=0.7]
  \node[shape=circle,draw=black] (A) at (0,0) {$w$};
  \node[shape=circle,draw=black] (B) at (2,2) {$v$};
  \node[shape=circle,draw=black] (E) at (-2,2) {$u$};
  \path (A) edge (B);
  \path (A) edge (E);
  \node (H) at (2.8,2)  {$p,r$};
  \node (J) at (-2.8,2)  {$p,q$};
 \end{tikzpicture}
 \caption{A model that teases apart standard disjunction and inquisitive disjunction on a singleton team. At the root $w$, $p\to q\lor r$ is true, but $p\to q\,{\protect\disji}\, r$ is not.}
 \label{fig:lor vs lori}
\end{figure}
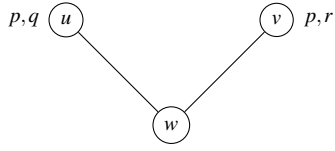

\subsection{Normal form}
\label{sub:normal form}

\noindent In Section \ref{sub:normal form classical} we have seen that the logic \lqd\ has a strong normal form result: every formula is equivalent to an inquisitive disjunction of standard formulas. This means that a formula can always be put in a form which brings all its ``inquisitive content'' to the surface layer. Perhaps surprisingly, this strong result carries over unmodified to the intuitionistic setting. Any formula $\phi\in\L$ is equivalent, in \ilqd\ as well as in \lqd, to the inquisitive disjunction of its resolutions, as given by Definition \ref{def:resolutions}. In particular, any formula is equivalent to an inquisitive disjunction of standard formulas.\footnote{It is important to note that the inquisitive normal form only concerns the interactions of inquisitive disjunction with the standard part of the language---not the interactions of standard connectives with each other. Indeed, one can check for each standard formula $\alpha$ we have $\R(\aa)=\{\aa\}$, and thus the inquisitive normal form theorem boils down to the trivial equivalence $\aa\equiv\aa$. 
This illustrates the fact that, although the normal form results allows us to bring all occurrences of inquisitive disjunction at the surface layer of the formula (i.e., to avoid occurrences of $\lori$ within the scope of other connectives) it does not impose any specific constraints on the form of the standard formulas occurring below this layer.}

\begin{theorem}[Inquisitive normal form]\label{theor:normal form}~\\
Let $\phi\in\L$ and $\R(\phi)=\{\aa_1,\dots,\aa_n\}$. Then $\phi\equiv_{\ilqd}\aa_1\lori\dots\lori\aa_n$. 
\end{theorem}

\begin{proof} 
The theorem is proved by induction on the complexity of $\phi$. We only give the proof of the induction step for standard disjunction $\lor$ and implication.

Case $\phi=\psi\lor\chi$. By induction hypothesis, we have $\psi\equiv_{\ilqd}\Lori \R(\psi)$ and $\chi\equiv_{\ilqd}\Lori \R(\chi)$. It is easy to verify that for all formulas $\theta,\delta_1,\delta_2$, the distributive law $\theta\lor(\delta_1\lori\delta_2)\equiv_{\ilqd}(\theta\lor\delta_1)\lori(\theta\lor\delta_2)$ holds. By applying this law, we obtain $\psi\lor\chi\equiv_{\ilqd}\Lori\{\alpha\lor\beta\mid \alpha\in \R(\psi)\text{ and }\beta\in \R(\chi)\}=\Lori \R(\psi\lor\chi)$.

Case $\phi=\psi\to\chi$. We have
\begin{align*}
\psi\to\chi&\equiv_{\ilqd} \Lori\R(\psi)\to\Lori \R(\chi)\quad\quad\quad\text{(by induction hypothesis)}\\
&\equiv_{\ilqd}\bigwedge_{\alpha\in \R(\psi)}(\alpha\to\Lori \R(\chi))\\
&\hfill\quad\quad\quad\quad\quad\quad\quad\quad\quad\quad\quad(\text{since }\theta_1\lori\theta_2\to\delta\equiv_{\ilqd} (\theta_1\to\delta)\land(\theta_2\to\delta))\\
&\equiv_{\ilqd}\bigwedge_{\alpha\in \R(\psi)}\mathop{\DISJI}_{\beta\in \R(\chi)}(\alpha\to\beta)\quad\quad\quad\text{(by Theorem \ref{theor:internal split})}\\
&\equiv_{\ilqd}\DISJI\;{\{\!\bigwedge_{\alpha\in \R(\psi)}(\alpha\to f(\alpha))\mid f:\R(\psi)\to\R(\chi)\}}\\
&\hfill\quad\quad\quad\quad\quad\quad\quad\quad\quad\quad\quad\quad\text{(since $\theta\wedge(\delta_1\lori\delta_2)\equiv_{\ilqd}(\theta\wedge\delta_1)\lori(\theta\wedge\delta_2)$)}\\
&\equiv_{\ilqd}\Lori\R(\psi\to\chi)
\end{align*}

\end{proof}

\noindent Notice that this normal form result is immediately inherited by any strengthening of \inqi, since a stronger logic is bound to give rise to a coarser notion of equivalence. In particular, the analogous result for the classical case, Theorem \ref{theor:normal form cl} falls out as a corollary. It should be noted, however, that in stronger logics the normal form result can be given alternative formulations which are not valid in \ilqd. For instance, the first normal form result for classical inquisitive logic, established by Ciardelli and Roelofsen \cite{CiardelliRoelofsen:11jpl},  represents each formula as an inquisitive disjunction of negations. An analogous result fails for \ilqd: it is easy to see, e.g., that an atom $p$ is not equivalent to any inquisitive disjunction of negations in \ilqd.

Using the normal form result, we can show that \inqi\ has the finite model property.


\begin{theorem}[Finite Model Property]\label{fmp}
If $\not\models\phi$, then there exists a finite Kripke model $M$ and finite team $t$ such that $M,t\not\models\phi$.
\end{theorem}
\begin{proof}
For any $\phi\in\Lp$, by Theorem \ref{theor:normal form} we have $\phi\equiv_\inqi\aa_1\lori\dots\lori\aa_n$, where $\R(\phi)=\{\aa_1,\dots,\aa_n\}$. If $\not\models_\inqi\phi$, then $\not\models_\inqi\alpha_i$ for all $1\leq i\leq n$. Since each $\alpha_i$ is a standard formula, by Theorem \ref{conservativity_ipl}, $\not\models_{\ipl}\alpha_i$ for each $i$. By the finite model property of \ipl, for each $1\leq i\leq n$ and by Proposition \ref{prop:intuitionistic truth-conditions}, there exists a finite Kripke model $M_i$ and a world $w_i$ such that $M_i,w_i\not\models\alpha_i$.
Consider now the finite model $M=M_1\uplus\dots\uplus M_n$ and the finite team $t=\{w_1,\dots,w_n\}$. For each $i\in I$, Proposition \ref{prop:locality} gives $M,w_i\not\models\alpha_i$, which implies $M,t\not\models\alpha_i$ by persistency. Thus, $M,t\not\models\aa_1\lori\dots\lori\aa_n$, whence by Theorem \ref{theor:normal form} we conclude $M,t\not\models\phi$.
\end{proof}

\noindent An important consequence of this result is that, together with the completeness result that we are going to establish later on, it implies that \inqi\ is a decidable logic: there is a recursive procedure that, given a $\phi\in\L$ decides in a finite time whether $\models\phi$ holds. This procedure consists simply in searching in parallel for a proof of $\phi$ and for a finite counter-model to $\phi$. By completeness and the finite model property, this search is guaranteed to terminate in a finite time, yielding an answer to the question.

\begin{corollary} \inqi\ is decidable.
\end{corollary}

\subsection{Truth-conditionality and standard formulas}
\label{sub:truth-conditional}

\noindent Proposition \ref{prop:truth-conditionality} ensures that in \ilqd all standard formulas are truth-conditional. We saw in Section \ref{sub:adding questions} that adding inquisitive disjunction allows us to express questions, i.e., formulas that are not truth-conditional, such as $?p$. 
Clearly, then, the presence of inquisitive disjunction gives us some extra expressive power over the standard language: formulas like $?p$ are not equivalent to any standard formula. However, one may wonder whether the presence of inquisitive disjunction also allows us to express more statements beyond what intuitionistic logic can express. 

In the classical case, we saw that the answer is negative: standard formulas are representative of all truth-conditional formulas. We proved this by associating every formula $\phi$ with a standard formula $\phi^s$ obtained by replacing each inquisitive disjunction symbol $\lori$ with $\lor$, and proving that this preserves the formula's truth-conditions. We might expect the same  to be true in \ilqd, since the formulas $\phi\lor\psi$ and $\phi\lori\psi$ have exactly the same truth-conditions: both are true at a world if one of the disjuncts is true. Interestingly, this is not the case: a formula $\phi$ and the formula $\phi^s$ do \emph{not}, in general, have the same truth-conditions. In the intuitionistic case, even single worlds can discriminate between the two disjunction operators. To see this, notice that at the root of the model in Fig.\ \ref{fig:lor vs lori}, $p\to q\lor r$ is true, but $p\to q\lori r$ is not.




Technically, this difference between $\lor$ and $\lori$ arises because, even if we start from a singleton team, in the intuitionistic setting implication can lead us to consider non-singleton extensions of this team. This does not happen in the classical case: in \lqd, extensions are just subsets, and thus a singleton state can never have a non-singleton extension. This diagnosis leads us to suspect that the only case in which a truth-conditional difference between the two disjunctions can arise is when these disjunctions are embedded within an implication. The next proposition ensures that this is indeed the case. 

\begin{proposition}\label{prop:when replacement works}~\\
Suppose no occurrence of $\lori$ is within the scope of an implication in $\phi$. Then, $\phi$ and $\phi^s$ have the same truth-conditions. 
\end{proposition}

\begin{proof} For any $\psi,\chi\in\L$, the formulas $\psi\lor\chi$ and $\psi\lori\chi$ have the same truth-conditions: both are true when at least one of the disjuncts is true.
Using this fact, a simple proof by induction on the structure of $\phi$ suffices to establish the claim. 
\end{proof}

\noindent
Even though $\lori$ cannot generally be replaced by $\lor$ preserving truth-conditions, it turns out that any truth-conditional formula in \ilqd\ is still equivalent to a standard formula. To see this, we first associate with each formula $\phi$ a standard formula that has the same truth-conditions.

\begin{theorem}~\\
Let $\phi\in\L$ and $\R(\phi)=\{\aa_1,\dots,\aa_n\}$. Then, $\phi$ has the same truth conditions as $\aa_1\lor\dots\lor\aa_n$.
\end{theorem}

\begin{proof} By Theorem \ref{theor:normal form}, $\phi\equiv\aa_1\lori\dots\lori\aa_n$. In the formula $\aa_1\lori\dots\lori\aa_n$, no inquisitive disjunction is within the scope of an implication. So, Proposition \ref{prop:when replacement works} applies, ensuring that $\aa_1\lori\dots\lori\aa_n$ and $\aa_1\lor\dots\lor\aa_n$ have the same truth-conditions.
\end{proof}

\noindent If a formula $\phi\in\L$ is truth-conditional, the previous theorem implies that it is equivalent with the standard disjunction of its resolutions, which is a standard formula. Thus, we get the following corollary. 

\begin{corollary}\label{truth_cond=standard}~\\
For any $\phi\in\L$, $\phi$ is truth-conditional$\iff \phi\equiv\aa$ for some standard formula $\aa$.
\end{corollary}

\noindent
This shows that, as far as statements are concerned, our logic is not more expressive than intuitionistic logic. One can also prove that any truth-conditional formula $\phi$ is equivalent to a specific one of its resolutions. To see this, notice that by the normal form theorem, $\phi\models\aa_1\lori\dots\lori\aa_n$; since $\phi$ is truth-conditional, the split property applies, giving $\phi\models\aa_i$ for some $i\le n$. On the other hand, it follows immediately from the normal form result that $\aa_i\models\phi$. Hence, $\phi\equiv\aa_i$.

Before turning to the topic of axiomatization, it is worth pointing out another way in which our logic \ilqd\ departs from \lqd\ and more generally from classical versions of inquisitive logic. In Section \ref{sub:logic classical}, we saw that in \lqd\ the double negation law is the hallmark of truth-conditionality, in the sense that $\phi$ is truth-conditional iff $\phi\equiv\neg\neg\phi$ (Proposition \ref{sub:normal form classical}). This no longer holds in the intuitionistic setting: for a simple example, an atom $p$ is truth-conditional, but $p\not\equiv\neg\neg p$. What remains true in \inqi\ is that negations always are truth-conditional. Since $\neg\phi$ abbreviates $\phi\to\bot$, this is a particular case of the following, more general fact.

\begin{proposition} If $\psi$ is truth-conditional, then so is $\phi\to\psi$ for every formula $\phi$. 
\end{proposition}

\begin{proof}
Consider a model $M$ and a team $t$, and suppose $M,w\models\phi\to\psi$ for all $w\in t$. Let $t'\subseteq R[t]$ be such that $M,t'\models\phi$. Now take any $v\in t'$: by persistency, we must have $M,v\models\phi$. Since $t'\subseteq R[t]$ we must have $wRv$ for some $w\in t$. Since $M,w\models\phi\to\psi$, it follows that $M,v\models\psi$. Since this is true for any $v\in t'$ and since $\psi$ is truth-conditional, we can conclude $M,t'\models\psi$. Thus, $M,t\models\phi\to\psi$.
\end{proof}

%

\subsection{Axiomatization}
\label{sub:axiomatization}

In Section \ref{sub:logic classical} we presented a sound and complete natural deduction system for the logic \lqd. The rules of this system are given in Figure \ref{fig:proof system lqd}. This system includes a complete natural deduction system for classical propositional logic. In particular, for standard formulas we have the double negation elimination rule. This rule is no longer sound in our setting, since $\neg\neg p\not\entilqd p$. However, we are now going to show that simply dropping the double negation elimination rule yields a sound and (strongly) complete proof system for our logic \ilqd. 

\begin{definition} Let $\Phi\cup\{\psi\}\subseteq\L$. We write $\Phi\prilqd\psi$ if there is a proof of~$\psi$ from assumptions in $\Phi$ which uses the inference rules in Figure\ \ref{fig:proof system lqd} except for \dn. We write $\phi\preqilqd\psi$ in case $\phi\prilqd\psi$ and $\psi\prilqd\phi$.
\end{definition}

\noindent As usual, proving that the given proof system is sound is just a matter of checking that each inference rule is sound. We omit the straightforward proof.

\begin{proposition}[Soundness] $\Phi\prilqd\psi\;\Rightarrow\;\Phi\entilqd\psi$
\end{proposition}

\noindent We are now going to prove that the given proof system is also strongly complete for \ilqd. To show this, we first need to establish some lemmata. The first lemma tells us that our system can prove the equivalence between a formula and its normal form.

\begin{lemma}[Provable normal form]\label{lemma:provable normal form}~\\ 
If $\phi\in\L$ and $\R(\phi)=\{\aa_1,\dots,\aa_n\}$, $\;\phi\preqilqd\aa_1\lori\dots\lori\aa_n$.
\end{lemma}

\begin{proof} By induction on the structure of $\phi$. It suffices to check that each of the equivalences used in the proof of  Theorem \ref{theor:normal form} to bring a formula in inquisitive normal form can be deduced in the proof system.
\end{proof}

\noindent
The second lemma states that if $\psi$, together with some other assumption $\Phi$, fails to derive $\chi$, then this can be traced to the fact that some specific resolution $\aa\in\R(\psi)$ together with $\Phi$ fails to derive $\chi$.

\begin{lemma}\label{lemma:specification lemma} If $\,\Phi,\psi\not\vdash\chi\,$, then $\,\Phi,\aa\not\vdash\chi\,$ for some $\,\aa\in\R(\psi)$.
\end{lemma}

\begin{proof} We will prove the contrapositive: if $\Phi,\aa\vdash\chi$ for all $\aa\in\R(\psi)$, then $\Phi,\psi\vdash\chi$. Let $\R(\psi)=\{\aa_1,\dots,\aa_n\}$. The rule $(\lori\!\textsf{e})$ ensures that, if we have $\Phi,\aa_i\vdash\chi$ for $1\le i\le n$, we also have $\Phi,\aa_1\lori\dots\lori\aa_n\vdash\chi$. Since the previous lemma gives $\psi\vdash\aa_1\lori\dots\lori\aa_n$, we also get $\Phi,\psi\vdash\chi$.
\end{proof}

\noindent The next lemma extends this result from a single assumption to the whole set.

\begin{lemma}[Traceable deduction failure]\label{lemma:global specification lemma}~\\ 
If $\Phi\not\vdash\psi$, then $\Gamma\not\vdash\psi$ for some set of standard formulas $\Gamma$ which contains a resolution for each $\phi\in\Phi$.
\end{lemma}

\begin{proof} Let us fix an enumeration of $\Phi$, say $(\phi_n)_{n\in\mathbb{N}}$. We are going to define a sequence  $(\aa_n)_{n\in\mathbb{N}}$ of standard formulas such that, for all $n\in\mathbb{N}$:
\begin{enumerate}
\item $\aa_n\in\R(\phi_n)$;
\item $\{\aa_i\,|\,i\le n\}\cup\{\phi_i\,|\,i> n\}\not\vdash\psi$. 
\end{enumerate}
We apply the previous lemma inductively. Suppose we have defined $\aa_i$ for $i<n$ and let us proceed to define $\aa_n$. The induction hypothesis tells us that $\{\aa_i\,|\,i<n\}\cup\{\phi_i\,|\,i\ge n\}\not\vdash\psi$, that is, $\{\aa_i\,|\,i<n\}\cup\{\phi_i\,|\,i> n\},\phi_n\not\vdash\psi$. By the previous lemma, we can find an $\aa_n\in\R(\phi_n)$ such that $\{\aa_i\,|\,i<{n}\}\cup\{\phi_i\,|\,i> n\},\aa_n\not\vdash\psi$. Hence, $\{\aa_i\,|\,i\le {n}\}\!\cup\!\{\phi_i\,|\,i> n\}\not\vdash\psi$, completing the inductive step.

Now let $\Gamma:=\{\aa_n\,|\,n\in\mathbb{N}\}$. By construction, $\Gamma$ is a set of standard formulas which contains a resolution of each $\phi\in\Phi$. We claim that $\Gamma\not\vdash\psi$. To see this, suppose towards a contradiction $\Gamma\vdash\psi$: then for some $n$ it should be the case that $\aa_1,\dots,\aa_n\vdash\psi$; but this is impossible, since by construction we have $\{\aa_1,\dots,\aa_n\}\cup\{\phi_{i}\,|\,i>n\}\not\vdash\psi$. Thus, we have $\Gamma\not\vdash\psi$, which completes the proof of the lemma.
\end{proof}

\noindent Now let $M_c^{\ipl}$ be the canonical Kripke model for intuitionistic logic (see, e.g., \citep{ZhaCha_ml}). We will prove strong completeness by showing that if $\Phi\not\vdash_\inqi\psi$, then we can find a team $t$ in the model $M_c^\ipl$ which supports all formulas in $\Phi$ but not $\psi$.

\begin{theorem}[Strong completeness]\label{th:strong completeness} $\Phi\entilqd\psi\;\Rightarrow\;\Phi\prilqd\psi$
\end{theorem}

\begin{proof} By contraposition, suppose $\Phi\not\prilqd\psi$. By Lemma \ref{lemma:global specification lemma} we have a set $\Gamma$ of standard formulas which contains a resolution of each $\phi\in\Phi$ and such that $\Gamma\not\vdash\psi$. Now let $\R(\psi)=\{\aa_1,\dots,\aa_n\}$: we must have $\Gamma\not\vdash\aa_i$ for $1\le i\le n$: for otherwise, using $(\lori\textsf{i})$ and Lemma \ref{lemma:provable normal form} we could conclude $\Gamma\vdash\psi$.

Now consider any $i\le n$: $\Gamma\cup\{\aa_i\}$ is a set of standard formulas, and $\Gamma\not\prilqd\aa_i$. Since our system includes a complete proof system for intuitionistic logic, and since truth for standard formulas coincides with truth in Kripke semantics, we know from the properties of the canonical model $M_c^\ipl$ that we can find a world $w_i$ in $M_c^{\ipl}$ such that $M_c^{\ipl},w_i\models\gamma$ for all $\gamma\in\Gamma$, but $M_c^{\ipl},w_i\not\models\aa_i$. 

Now let consider the team $t:=\{w_1,\dots,w_n\}$. We claim that $t$ supports all formulas in $\Phi$ but does not support $\psi$. First take any $\phi\in\Phi$. The set $\Gamma$ contains a resolution $\gamma\in\R(\phi)$. By construction, we know that each formula in $\Gamma$ is true at each world in $t$. Moreover, all formulas in $\Gamma$ are standard formulas, and therefore they are truth-conditional (Proposition \ref{prop:truth-conditionality}). This implies $M_c^{\ipl},t\models\gamma$. Now, the normal form result (Theorem \ref{theor:normal form}) implies that $\gamma\models_\inqi\phi$. Therefore, $M_c^{\ipl},t\models\phi$. Since $\phi$ was an arbitrary element of $\Phi$, the team $t$ supports all formulas in $\Phi$.

To see that $t$ does not support $\psi$, suppose towards a contradiction that $M_c^{\ipl},t\models\psi$. By the normal form theorem  this would mean that $M_c^{\ipl},t\models\aa_i$ for some $i\le n$. Since $w_i\in t$, by persistency we should also have $M_c^{\ipl},w_i\models\aa_i$, which is a contradiction.

We have thus found a team in the Kripke model $M_c^{\ipl}$ which supports all formulas in $\Phi$ but not $\psi$, witnessing that $\Phi\not\models_\inqi\psi$.
\end{proof}

\noindent Notice that all that we used about our proof system to obtain this result was that (i) it includes a proof system for intuitionistic logic and (ii) it is capable of proving the equivalence between a formula and its normal form. The first feature pins down the underlying logic of statements, while the latter completely captures the behavior of propositional questions relative to this underlying logic.

Also, notice that in terms of entailment, the only difference between \inqi\ and \lqd\ is the validity in \lqd\ of the double negation law for the standard language---in other words, the fact that the underlying logic of statements is classical in the case of \lqd, but intuitionistic in the case of \inqi. This shows that the logical features of questions and dependencies are quite modular with respect to the underlying logic of statements.


\subsection{The $\lor$-free fragment of \ilqd}

Consider the sub-logic \inqim\ obtained from \inqi\ by dropping the standard disjunction $\lor$ from the language. If we simply regard $\lori$ as the  ``official'' disjunction operator of the system, \inqim\ has a standard repertoire of connectives. Thus, it can  be meaningfully compared in terms of strength with ordinary logics, such as intuitionistic logic, classical logic, and other intermediate logics. Our axiomatization result for \inqi\ immediately shows that all principles of intuitionistic logic are valid for \inqim. Moreover, all principles of \inqim\ are valid in classical logic, for it is easy to see that any entailment which is not classically valid can be falsified in a Kripke model consisting of a single world. Thus, \inqim\ is a logic which is intermediate in strength between intuitionistic and classical logic. However, \inqim\ is not an intermediate logic in the usual sense, since its set of validities, just as in the case of \lqd\ and \inqi, is not closed under uniform substitution: for instance, the formula $(p\to q\lori r)\to (p\to q)\lori (p\to r)$ is valid in \inqim, although replacing $p$ by $q\lori r$ yields an obviously invalid formula. 

It turns out that our results above can easily be adapted to give a sound and complete axiomatization of \inqim. Such an axiomatization is obtained simply by dropping from the system given for $\prilqd$ all the rules dealing with $\lor$. This leaves us with what is essentially the standard natural deduction system for intuitionistic logic---with $\lori$ now in the role of intuitionistic disjunction---augmented with the split scheme \textsf{S}. Denoting this system by $\vdash_{\inqim}$, we have the following result.

\begin{theorem}[Soundness and completeness for \inqim]~\\
For any formulas $\Phi\cup\{\psi\}$ in the language of \inqim: $\Phi\models_{\inqim}\psi\iff\Phi\vdash_{\inqim}\psi$
\end{theorem}

\begin{proof} We just need to repeat for the $\lor$-free fragment the same proof given above for the whole language, and verify that none of the rules concerned with $\lor$ is ever needed to reach the conclusion. 
\end{proof}

\noindent Pun\v coch\'a\v r \cite{Puncochar:15generalization} discusses a family of logics, called \emph{G-logics}, which, like our logic \inqi, can be seen as variants of classical inquisitive logic where the underlying logic of statements is weaker than classical logic. The result we have just established tells us that the $\lor$-free fragment of our intuitionistic inquisitive logic \inqi\ coincides with the least element of this family, the logic that Pun\v cocha\v r denotes by \textsf{IL+H}.

\subsection{Translation of \lqd\ into \inqi}
\label{sub:negative translation}

Glivenko's theorem states that a formula $\phi$ is valid in classical propositional logic iff its double negation $\neg\neg\phi$ is valid in intuitionistic propositional logic. In fact, the map $\phi\mapsto\neg\neg\phi$ yields an embedding of \cpl\ into \ipl. Does the same relation hold between the classical system \lqd\ and its intuitionistic counterpart \inqi? 

The answer is negative: for instance, the formula $?p$ is not valid in \lqd, but it is easy to see that its double negation $\neg\neg{?p}$ is valid in \inqi. Nevertheless, an embedding of \lqd\ into \inqi\ can be obtained if we take care of distributing the double negation translation over the resolutions of a formula. That is, if $\R(\phi)=\{\aa_1,\dots,\aa_k\}$, we make the following definition:
$$\phi^n:=\neg\neg\aa_1\lori\dots\lori\neg\neg\aa_k$$
\noindent If $\aa$ is a standard formula, $\aa^n$ coincides with the double negation $\neg\neg\aa$. Glivenko's theorem generalizes to the inquisitive setting in the following form.

\begin{proposition} $\models_{\lqd}\phi\iff\models_{\inqi}\phi^n$
\end{proposition}

\begin{proof} Recall that the disjunction property holds for inquisitive disjunction $\lori$ in both systems \lqd\ and \inqi. Suppose $\models_{\lqd}\phi$: since $\phi\equiv_{\lqd}\aa_1\lori\dots\lori\aa_k$, the disjunction property yields $\models_{\lqd}\aa_i$ for some $i\le k$. But $\aa_i$ is a standard formula, and on standard formulas \lqd\ coincides with \cpl. So $\aa_i$ is classically valid, which by Glivenko's theorem implies that $\neg\neg\aa_i$ is valid in \ipl. Since \inqi\ coincides with \ipl\ on standard formulas, we have $\models_{\inqi}\neg\neg\aa_i$, which immediately implies $\models_{\inqi}\phi^n$.  The argument for the converse direction is analogous.
\end{proof}

\noindent
In fact, it is easy to strengthen this result to show that $(\cdot)^n$ gives a translation of \lqd\ into \inqi, in the following sense. 

\begin{proposition} Let $\Phi\cup\{\psi\}\subseteq\L$ and let $\Phi^n=\{\phi^n\mid \phi\in\Phi\}$. Then: 
$$\Phi\models_{\lqd}\psi\iff \Phi^n\models_\inqi\psi^n$$
\end{proposition}

\noindent What this result tells us is that the Lindenbaum-Tarski algebra of the logic \lqd\ embeds via $(\cdot)^n$ into the one of the logic \inqi. Thus, \lqd\ can be identified with a fragment of \inqi, namely, the fragment consisting of formulas which are equivalent to an inquisitive disjunction of negations.

An interesting question that remains is whether a translation of \lqd\ into \inqi\ can be defined directly by induction on the structure of the formula $\phi$, in analogy with the G\"odel-Gentzen negative translation of \cpl\ into \ipl. Interestingly, this does not appear to be the case, but we will not try to provide an impossibility proof here.

\subsection{Translation of \inqi into downward closed modal team logic}
\label{sub:godel translation}

G\"{o}del's \Sf\ translation embeds \ipl into the classical modal logic \Sf---the logic of Kripke frames with a reflexive and transitive accessibility relation. It is easy to generalize this translation to an embedding from \inqi into \Sf modal dependence logic \cite{Vaananen:08} extended with intuitionistic connectives, also called downward closed modal team logic \MT in \cite{Yang:16}.

Let us briefly recall the basics of \MT. We refer the reader to \cite{Yang:16} for further information.
\begin{definition}[Language \Ld]\label{def:mid language}~\\
The set \Ld of formulas of \MT is defined inductively as follows:%
$$\phi\;::=\;p\mid\bot\mid\dep(p_1,\dots,p_n,q)\mid\phi\land\phi\mid\phi\lor\phi\mid\phi\to\phi\mid\phi\lori\phi\mid\Box\phi\mid\Diamond\phi$$%
\end{definition}

\noindent
Formulas of \MT are evaluated on modal Kripke models, i.e., triples $M=\<W,R,V\>$, where $W$ is a set, $R$ is a binary relation on $W$, and $V:W\times\P\to\{0,1\}$ is a valuation function.
Exactly as with \inqi, a team in a modal Kripke model $M=\<W,R,V\>$ is a set of worlds $t\subseteq W$. However, in this setting a team $t'$ is said to be an $R$-successor of a team $t$, notation $tRt'$, in case $t'\subseteq R[t]$ and for all $w\in t$, $R[w]\cap t'\neq \emptyset$.

\begin{definition}[Team semantics for \MT]\label{defsupportsemantics_mt}~\\ 
Let $M=\<W,R,V\>$ be a modal Kripke model. The relation of satisfaction between teams $t$ and formulas $\phi\in\Ld$ is defined inductively like the support relation defined in Definitions \ref{defsupportsemantics} and \ref{defsupportsemantics_indisj} except for the following clauses for dependence atoms, implication and modalities:\footnote{It is easy to see that in \MT, the dependence atoms are definable by means of the remaining connectives in exactly the way we described in Section \ref{sub:dependence}, i.e., by taking $=\!\!(p_1,\dots,p_n,q)$ to be an abbreviation for ${?p_1}\land\dots\land{?p_n}\to{?q}$. Thus, dependence atoms are not really needed in the syntax of \MT. Also, notice that our translation does not make use of $\Diamond$, so that our translation can also be viewed as an imbedding into the $\Diamond$-free fragment of \MT.}
\begin{itemize}
\item $M,t\models\dep(p_1,\dots,p_n,q)\iff\forall w,w'\in t$, $V(w,p_i)=V(w',p_i)$ for all $1\leq i\leq p$ implies $V(w,q)=V(w',q)$
\item $M,t\models\phi\to\psi\iff\forall t'\subseteq t$, $\,M,t'\models\phi$ implies $M,t'\models\psi$
\item $M,t\models\Diamond\phi \iff\exists t'$ such that $tRt'$ and $M,t'\models\phi$
\item $M,t\models\Box\phi \iff\forall t'$, $tRt'$ implies $M,t'\models\phi$
\end{itemize}
\end{definition}
\noindent
\MT-formulas enjoy the same empty team property as \inqi, and the following version of the persistency property:
\begin{itemize}
\item Persistency property: if $M,t\models\phi$ and $s\subseteq t$, then $M,s\models\phi$
\end{itemize}
In view of persistency, the semantic clause for $\Box\phi$ can also be written as: 
\begin{itemize}
\item $M,t\models\Box\phi \iff M,R[t]\models\phi$
\end{itemize}
\noindent
We call a formula $\phi\in\Ld$ that does not have any occurrences of dependence atoms or $\lori$ a \emph{standard formula}. As with \inqi, standard formulas  of \MT are also truth conditional in the same sense (see Definition \ref{truth_cond_def}). 
There is also a disjunctive normal form for \MT that is very similar to the one for \inqi: every \MT-formula is equivalent to a formula of the form $\alpha_1\lori\cdots\lori\alpha_n$, where each $\alpha_i$ is a standard formula. A natural deduction system that is sound and  (strongly) complete with respect to \MT can be found in \cite{Yang:16}. Without going into detail we remark that using the disjunctive normal form of \MT, by a very similar argument to that of Theorem \ref{th:strong completeness}, one can prove that the system of \MT extended with the usual $\mathsf{T}$ axiom $\Box\alpha\to\alpha$ and $\mathsf{4}$ axiom $\Box\alpha\to\Box\Box\alpha$ for standard formulas $\alpha$, denoted \Sf-\MT, is sound and (strongly) complete with respect to the class of reflexive and transitive modal Kripke frames\footnote{It is easy to check that reflexive and transitive frames validate the $\mathsf{T}$ and $\mathsf{4}$ axioms for standard formulas. For the strong completeness, assume $\Phi\not\vdash_{\Sf-\MT}\psi$. Suppose $\psi\dashv\vdash\alpha_1\lori\cdots\lori\alpha_n$, and $\phi_i\dashv\vdash\gamma_{i1}\lori\cdots\lori\gamma_{in_i}$ for every $\phi_i\in \Phi$. Then, $\Gamma\not\vdash_{\Sf-\MT}\psi$ for some set $\Gamma$ consisting of one disjunct $\gamma_{ii_k}$ for each formula $\phi_i$ in $\Phi$, and therefore $\Gamma\not\vdash_{\Sf-\MT}\alpha_i$ for all $\alpha_i$. For each $i$, the canonical model $M_c^{\Sf}$ of \Sf (whose frame is reflexive and transitive) contains a witness $w_i$ such that $M_c^{\Sf},w_i\models\gamma$ for all $\gamma\in \Gamma$, while $M_c^{\Sf},w_i\not\models\alpha_i$. For $t=\{w_1,\dots,w_n\}$, we get $M_c^{\Sf},t\models\gamma$ for all $\gamma\in \Gamma$, and $M_c^{\Sf},t\not\models\psi$.}, called \Sf-modal Kripke frames. We write $\models_{\Sf-\MT}\phi$ if $\phi$ is valid in this semantics with respect to \Sf\ models, that is, if $M,t\models\phi$ for all \Sf-modal Kripke models $M$ and teams $t$. 
G\"odel's \Sf\ translation can then be generalized to a translation of \inqi\ into \Sf-\MT as follows.

\begin{definition}[\Sf\ translation of \inqi\ into \MT]~\\
We define a translation ${(\cdot)}^\Box:\Lp\to\Ld$ 
inductively as follows:
\begin{itemize}
\item $p^\Box=\Box p$
\item $\bot^\Box=\Box\bot$
\item $(\phi\wedge\psi)^\Box=\phi^\Box\wedge\psi^\Box$
\item $(\phi\lor\psi)^\Box=\phi^\Box\lor\psi^\Box$
\item $(\phi\lori\psi)^\Box=\phi^\Box\lori\psi^\Box$
\item $(\phi\to\psi)^\Box=\Box(\phi^\Box\to\psi^\Box)$
\end{itemize}
\end{definition}

\begin{theorem}
$\models_{\inqi}\phi\iff \models_{\Sf-\MT}\phi^\Box$
\end{theorem}
\begin{proof}
For every \Sf\ Kripke model $M=\<W,R,V\>$, define an intuitionistic Kripke model $\rho M=\<\rho W,\rho R,\rho V\>$ by taking
\begin{itemize}
\item $\rho W=\{w^c\mid w\in W\}$, where $w^c:=\{v\in W\mid wRv\text{ and }vRw\}$
\item $w^c\,\rho R\,v^c\iff wRv$
\item $\rho V(w^c,p)=1\iff M,w\models\Box p$
\end{itemize}
We leave it for the reader to verify (or see \S 3.9 in \cite{ZhaCha_ml}) that $\rho M$ is well-defined, and that if $M$ is an intuitionistic Kripke model,  then $M$, viewed as an \Sf-modal Kripke model, is isomorphic to $\rho M$. Letting $t^c=\{w^c\mid w\in t\}$, we claim that for any $\phi\in\Lp$: 
$$M,t\models_{\Sf-\MT}\phi^\Box\iff\rho M,t^c\models_{\inqi}\phi.$$ 

\noindent
The claim can be proved by induction on $\phi$. If $\phi=p$, then exploiting the definition of $\rho V$ and the fact that  $\Box p$ is standard and thus truth-conditional in \MT, we have:
\begin{eqnarray*}
\rho M, t^c\models_\inqi p & \iff & \forall w^c\in t^c: \rho V(w^c,p)=1\\
&\iff & \forall w\in t: M,w\models\Box p\\
&\iff & M,t\models\Box p 
\end{eqnarray*}
%
%
%
If $\phi=\bot$,  then
since $R$ is reflexive,
$$M,t\models_{\Sf-\MT}\Box\bot\iff R[t]=\emptyset\iff t^c=\emptyset\iff\rho M,t^c\models_{\inqi}\bot$$
For the only nontrivial inductive case $\phi=\psi\to\chi$, we have
\begin{align*}
&M,t\not\models_{\Sf-\MT} \Box(\psi^\Box\to\chi^\Box)\\
\iff& M,R[t]\not\models_{\Sf-\MT} \psi^\Box\to\chi^\Box\\
\iff& \exists s\subseteq R[t]\text{ s.t. }M,s\models_{\Sf-\MT} \psi^\Box\text{ and }M,s\not\models_{\Sf-\MT} \chi^\Box\\
\iff& \exists s^c\subseteq \rho R[t^c]\text{ s.t. }\rho M,s^c\models_{\inqi} \psi\text{ and }\rho M,s^c\not\models_{\inqi} \chi\\
&\quad\text{(by induction hypothesis, and since $s\subseteq R[t]$ iff $s^c\subseteq \rho R[t^c]$)}\\
\iff&\rho M,t^c\models_{\inqi}\psi\to\chi.
\end{align*}

We now complete the proof by applying the claim. If $\not\models_{\Sf-\MT}\phi^\Box$, then $M,t\not\models_{\Sf-\MT}\phi^\Box$ for some \Sf-modal Kripke model $M$, which by the claim implies that $\rho M,t^c\not\models_{\inqi}\phi$, and thereby $\not\models_{\inqi}\phi$.

 Conversely, if $\not\models_{\inqi}\phi$, then $M,t\not\models_{\inqi}\phi$ for some intuitionistic Kripke model $M$. The model $M$ can be viewed as an \Sf-modal Kripke model that is isomorphic to $\rho M$. Thus, by the claim, $M,t\not\models_{\Sf-\MT}\phi^\Box$, from which we conclude $\not\models_{\Sf-\MT}\phi^\Box$.
\end{proof}


\section{Related work}
\label{sec:related work}

\noindent In this section, we discuss in some detail the relation between our proposal and the related work developed by Pun\v coch\'a\v r in \cite{Puncochar:15generalization} and \cite{Puncochar:16algebras}.  Broadly speaking, Pun\v coch\'a\v r's aim is similar to the one we pursued, namely, to investigate how questions may be added, in the style of inquisitive semantics, to a propositional logic which is weaker than classical logic. In particular, in the two papers cited above Pun\v coch\'a\v r proposes two ways of generalizing inquisitive logic to a setting in which the underlying logic of statements is an intermediate logic, including intuitionistic logic (a further generalization to sub-structural logics is pursued in \cite{Puncochar:17}). 

Our proposal converges with Pun\v coch\'a\v r's in some important respects: most importantly, the intuitionistic inquisitive  logic that emerges from our system \inqi\ is a syntactic fragment of the logic that arises from the approach of \cite{Puncochar:16algebras} (where the language also contains a modal operator) and in turn, as pointed out above, it contains as a fragment the minimal element of the family of inquisitive logics studied in \cite{Puncochar:15generalization}.

At the same time, our work here differs from Pun\v coch\'a\v r's in its technical workings as well as in its conceptual focus. Let us first consider the technical differences. In \cite{Puncochar:15generalization}, the weakening of the base logic is obtained by assuming that only certain teams are available as information states. That is, the semantics is based on structures called \emph{spaces of information states}, which consist of a universe $W$ of worlds equipped with a designated family $I$ of subsets of $W$. In \cite{Puncochar:16algebras}, a similar but more abstract approach is taken: information states are treated as primitives, rather than as sets of worlds, and an algebraic structure on the space of information states is assumed; more precisely, the semantics is based on structures called \emph{algebras of information states}, which are join semilattices with a bottom element. In both cases, the semantics is a natural adaptation of the one given by Definition \ref{def:support classical}. 

In this paper, we have taken a different approach: information states are still viewed as arbitrary sets of worlds, but the worlds themselves, i.e., the actual states of affairs, may be partially defined in some respects; this means that an information state can be extended not only by ruling out some worlds, but also by making some of the worlds more defined. 
Technically, this is achieved by equipping the space of worlds with a binary relation $R$ that encodes when a state refines another---just as in standard Kripke semantics for intuitionistic logic. Here, too, the semantic clauses for the connectives remain essentially the same as in the classical case. Thus, our approach provides a different and somewhat more standard semantic framework that allows us to interpret both statements and questions in an intuitionistic logical setting.


This technical difference is also connected to a difference in focus between Pun\v coch\'a\v r's work and our own. 
Pun\v coch\'a\v r's main aim is to generalize the way that inquisitive logic builds on classical logic, and to show how many results obtained for inquisitive logic admit much more general counterparts. Thus, the scope of his work is very broad, and the focus is mostly on the logics themselves and the relations between them. By contrast, our focus here has been narrower---we have zoomed in on intuitionistic logic as a logic of statements---but our interest has been as much in the semantics itself as in the logic that arises from it. Our starting point was the semantic analysis of questions and dependency in a classical propositional framework, as given by propositional inquisitive and dependence logic; we have looked at how this analysis can be extended to a framework that does not assume the law of excluded middle, and where states of affairs are partial by default. This lead us to study in detail the features of inquisitive logic when worlds are embedded in an intuitionistic Kripke model (to the best of our knowledge, this is also the first study of a team-based semantics in this setting). Moreover, we have looked in detail at how the  analysis of propositional dependencies given by inquisitive and dependence logic extends to the intuitionistic setting, and at the interesting issues that arise in connection with the possible lack of a definite truth-value for a statement at a world. 

We think that our approach provides novel insight into the workings of questions and dependency in the intuitionistic setting, and that our results contribute to the exploration of the landscape of non-classical inquisitive logics which was initiated by Pun\v coch\'a\v r's work.

\section{Conclusion and further work}

The move from a world-based semantics to a team-based one---from states of affairs to states of information---allows inquisitive and dependence logic to broaden the scope of classical logic, bringing questions and dependencies into play. In this paper, we have seen that this move is possible also in the intuitionistic setting. In this case, we are dealing with information about states of affairs which can themselves be partial---i.e., not defined in every respect. Formally, this means dealing with teams of worlds embedded within an intuitionistic Kripke model---or, equivalently, with sets of rooted Kripke models. The semantics remains exactly the same as in the classical case, except that now, extending a state does not just amount to eliminating some possibilities, but also to refining some of them in one or more ways. Just like in the world-based setting, the classical system is obtained as a special case by restricting the intuitionistic semantics to discrete Kripke models, models in which worlds are already complete and never have proper refinements.

We have seen that many key features of propositional inquisitive logic are preserved in the intuitionistic setting. In particular, the disjunctive normal form theorem carries over unmodified, and a sound and complete axiomatization is obtained simply by dropping the double negation rule for statements. In other words, we have seen that the only difference between the classical and the intuitionistic version of inquisitive logic lies in the underlying logic of statements, while the relation between statements and questions is the same in both cases. 

From the perspective of the semantics, on the other hand, we find many interesting novelties with respect to the classical case. For instance, in the intuitionistic case, even simple polar questions have non-trivial presuppositions---i.e., cannot always be truthfully resolved. Dependencies may fail to hold even with respect to singleton teams.
Any formula which is not logically valid can be falsified relative to a single world. The disjunctions $\lor$ and $\lori$ have the same truth-conditional behavior, and the disjunction property holds for both, yet they have different logical properties.

Clearly, many interesting issues must be left for future work. We briefly mention some of them here. First, while in this paper we have restricted ourselves to a propositional language, our approach can be extended straightforwardly to the first-order case by working with Kripke models for intuitionistic predicate logic. 
The predicate logic case is very interesting from our point of view since, as shown in \cite{Ciardelli:16} for the classical setting, it allows us to capture many important classes of questions (e.g., the question $\forall x{?Px}$ of which individuals have property $P$) and dependencies (e.g., the formula $\forall x{?Px}\to \forall x{?Qx}$ holds if the extension of property $Q$ is determined by the extension of property $P$).

Remaining within propositional logic, it would be interesting to investigate more precisely the expressive power of \inqi: what properties of teams in a Kripke model can be expressed by means of a formula in \inqi? Similar issues have been addressed for modal dependence logic in \cite{KontinenMulleerSchnoorVollmer15}, and for inquisitive modal logic in \cite{CiardelliOtto:17}, which may provide a good starting point to answer the question for \inqi.

In a slightly different direction, it would also be interesting to look at the issue of frame definability in \inqi. Clearly, if a standard formula defines a certain frame class in \ipl, then this formula still defines the same class in \inqi. At the same time, however, some frame classes which cannot be characterized in \ipl\ can now be characterized with the help of inquisitive formulas: for instance, $?p$ characterizes the class of singleton frames. Recent work on frame definability in the context of modal dependence logic \cite{SanoVirtema:15,SanoVirtema:16} might provide a handle on this question.

Finally, the work done in this paper could be taken further by studying questions and dependency in the context of other non-classical logics besides intuitionistic logic. This research direction is explored in recent work by Pun\v coch\'a\v r \citep[][unpublished]{Puncochar:17} which considers a broad range of substructural logics. Other cases that seem worth investigating include Kleene's three-valued logic or other $n$-valued logics.


%

\newcommand{\SortNoop}[1]{}



\end{document}